\documentclass[11pt]{article}
\usepackage[letterpaper,margin=1.00in]{geometry}
\usepackage{graphicx}
\usepackage{amsmath,amssymb}
\usepackage{enumitem}
\usepackage{framed}
\usepackage{float}

\usepackage{epstopdf}
\usepackage{todonotes}

\usepackage{tikz}
\usetikzlibrary{shapes,arrows}
\usepackage{todonotes}

\usepackage{pgfplots}
\usetikzlibrary{matrix}
\usepgfplotslibrary{groupplots}
\pgfplotsset{compat=newest}

\usepackage[titletoc]{appendix}

\usepackage[numbers]{natbib}
\usepackage{graphicx}
\usepackage{pdfsync}
\usepackage{color}
\usepackage{bbm}
\usepackage{mdwlist}
\usepackage{amsmath,amssymb}
\usepackage{enumitem} 
\usepackage{outline}
\usepackage{xspace}
\usepackage{floatflt}
\usepackage{rotating}

\usepackage{multirow}

\graphicspath{{./images/}}    

\usepackage{amsmath,amssymb,amsthm}		
\usepackage{xcolor}

\newcommand{\R}{\mathbb R}					
\newcommand{\st}{\mbox{s.t.}}				

\newtheorem{theorem}{Theorem}[section]
\newtheorem{definition}[theorem]{Definition}
\newtheorem{lemma}[theorem]{Lemma}
\newtheorem{assumption}[theorem]{Assumption}
\newtheorem{corollary}[theorem]{Corollary}
\newtheorem{remark}[theorem]{Remark}

\usepackage{algorithm}
\usepackage[noend]{algpseudocode}

\usepackage{graphicx}
\usepackage{epstopdf} 

\begin{document}

\title{Solving Quadratic Multi-Leader-Follower Games\\ by Smoothing the Follower's Best Response}
\author{Michael Herty, Sonja Steffensen, and Anna Th\"unen}
\date{\today}
\maketitle

\begin{abstract}

We derive Nash equilibria for a class of quadratic multi-leader-follower games using the nonsmooth best response function.
To overcome the challenge of nonsmoothness, we pursue a smoothing approach resulting in a reformulation as a smooth Nash equilibrium problem.
The existence and uniqueness of solutions are proven for all smoothing parameters.
Accumulation points of Nash equilibria exist for a decreasing sequence of these smoothing parameters and we show that these candidates fulfill the conditions of s-stationarity and are Nash equilibria to the multi-leader-follower game.
Finally, we propose an  update on the leader variables for efficient computation and numerically compare nonsmooth Newton and subgradient methods.
	
\medskip \noindent
{\bf Keywords: }Multi-Leader-Follower Games, Nash Equilibria,  Game Theory, Equilibrium Problems with Equilibrium Constraints

\medskip \noindent
{\bf AMS-MSC2010:}  91A06, 91A10, 90C33, 91A65, 49J52

\end{abstract}

\section{Introduction and Background }\label{S:intro}
The multi-leader-follower game (MLFG) is a particular class of problems in classical game theory \textcolor{black}{that represents a generalization} of the Stackelberg game which includes a single leader.
These models serve as an analytical tool for studying the  strategic behavior of noncooperative individuals.
In particular, the individuals (so-called players)  are divided into two groups, namely  leaders and  followers, according to their position in the game.
Mathematically, this yields a hierarchical Nash game, where further minimization problems appear in the participants' optimization problems as constraints.
An equilibrium is then given by a multistrategy vector of all players, where  no player has the incentive to change his chosen strategy unilaterally.

Most recently, such models have increasingly become the focus of interest among mathematicians as well as scientists in other fields such as operations research, robotics, and computer science \cite{Aussel16,Wang08,Koh10}.
However, compared to the knowledge of other classical game models, little is known   so far  concerning existence and uniqueness theory as well as suitable numerical solution methods.

\textcolor{black}{The structure of MLFGs is related to equilibrium problem with equilibrium constraints (EPEC) provided that the optimization problems of the followers can be replaced by the corresponding optimality conditions.}
The admissibility of this approach is discussed in \cite{Aussel2018}.
Recently,  the competition on the  electric power market is described by EPEC \cite{Allevi2018,Aussel17part1,Aussel17part2,henrion12,HR07}.
Here, the power generators and consumers are the leaders who bid their cost and utility functions.
The single follower plays the role of an independent system operator coordinating dispatch and minimizing  social costs subject to network constraints.

So far, there exist only a few recent theoretical results for general MLFGs or EPECs, analyzing the existence, the uniqueness and characterizations of equilibria: 
\textcolor{black}{
Early work by Sherali \cite{Sherali84} generalizes Stackelberg games in the setting of Cournot competition.
Existence theory of equilibria for identical leaders and sufficiency conditions for convex follower reaction functions is provided here.
Su \cite{Su2007} extended this work to a two-period forward market model and proposes an \textcolor{black}{SNCP} algorithm in \cite{SU2005}.
Later, further generalizations in the Cournot setting include the incorporation of stochasticity by DeMiguel and Xu  and the application to telecommunication industry \cite{DeMiguel2009}.}

\textcolor{black}{
In \cite{Kulkarni2014,Kulkarni2015}, Kulkarni and  Shanbhag introduce a more general setting for MLFGs and its reformulation with shared constraints.
The relation of Nash equilibria of the MLFG and weaker solution concepts of the reformulation are extensively discussed here, many results rely on potentiality of the players' objectives.
}

\textcolor{black}{Fukushima and Pang \cite{Pang2005} present a reformulation as generalized Nash equilibrium problem (GNEP).
Under suitable condtions, the equilibria of their GNEP reformulation can be characterized by a quasi-variational inequality.}
In another work, Hu and Ralph \cite{HR07} \textcolor{black}{apply a standard result }
for the existence of pure strategy Nash equilibria in a particular case of an MLFG related to an electricity market model.
Leyffer and Munson \cite{LeyfferMunson10} describe various reformulations in terms of  mathematical programs with equilibrium constraints (MPEC), nonlinear programs, and nonlinear complementarity problems.
\textcolor{black}{In a very recent paper \cite{Kim2019}, Kim and Ferris propose a reformulation of MLFG using extended mathematical programming which allows the usage of the complementarity problem solver PATH \cite{Dirkse1995,Ferris1999}.} 

More recent work has been done by Hu and Fukushima in \cite{HuFu2013}.
Therein, they discuss the existence of robust Nash equilibria of a class of quadratic MLFGs.
Further, they propose a uniqueness result for an MLFG with two leaders.

In this paper, we study a quadratic MLFG with  similarities to the model studied in \cite{HuFu2013} and generalize the follower's strategy set by allowing inequality constraints.
This modification translates  into equilibrium conditions or to a nonsmooth Nash game formulation.
The existence of Nash equilibria can be proven with suitable convexity assumptions.
Furthermore, we propose a smoothed Nash game formulation and prove uniqueness of Nash equilibria for an arbitrary number of players.
The main result of this article is that it can be shown that these smooth problems are indeed approximating Nash equilibria of the MLFG.
Besides the theoretical results, we propose an algorithm to numerically compute Nash equilibria.
Therein, we combine an update based on a Taylor expansion of the parameter-dependent solution and  the computation  of Nash equilibria of approximating problems.

%
%
%
 
This article is structured as follows.
In Section \ref{Sec:2}, we introduce the quadratic MLFG and develop an equivalent nonsmooth Nash equilibrium problem (NEP) for which we can prove existence of equilibria.
In Section \ref{Sec:3}, smoothing of the best response of the follower leads to a differentiable NEP formulation, which allows to characterize solutions by KKT conditions.
In Section \ref{Sec:4}, it is demonstrated that \textcolor{black}{the sequence} of Nash equilibria of the smooth problems converges also in the multipliers for decreasing smoothing parameters and that \textcolor{black}{S-stationarity} conditions are satisfied in the limit.
Eventually, this \textcolor{black}{S-stationary} point are proven to be Nash equilibria to the MLFG.
In Section \ref{Sec:5}, we introduce our general approach to computing Nash equilibria.
In Section \ref{Sec:6}, we present the numerical results of the proposed methods.

\section{Existence of Nash Equilibria for MLFGs}\label{Sec:2}

We consider an MLFG, where the follower's game is modeled by the optimization problem
\begin{equation}\label{follower}
\min\limits_{y\in \R^m} \frac{1}{2}y^\top Q_yy-b(x)^\top y \quad \st\quad y\geq l(x),
\end{equation}
\textcolor{black}{where $b_i,l_i:~\R^n\rightarrow\R$ are differentiable functions for $i=1,\dots,m$.
The matrix $Q_y\in\R^{m\times m}$ is  positive definite and diagonal, which  guarantees the existence of an explicit characterization of the minimizer, c.f. Lemma \ref{Th:bestresponse}.}
The leader problems are given for $\nu=1,\dots,N$ by
\begin{equation}\label{leader}
\min\limits_{x_\nu\in\R^{n_\nu}} \theta_\nu(x_\nu,x_{-\nu})=\frac{1}{2} x_\nu^\top Q_\nu x_{\nu}+c_\nu^\top x_\nu+a^\top \textcolor{black}{y(x)} \quad\st\quad x_\nu\in X_\nu,
\end{equation}
with nonempty, convex, and closed strategy sets $X_\nu$ which we assume to be described by smooth functions $g_\nu:\R^{n_\nu}\rightarrow\R^{m_\nu}$ such that $X_\nu=\left\{x_\nu\in\R^{n_\nu} | g_\nu(x_\nu)\leq0 \right\}$.
The quadratic objective $\theta_\nu$  is strictly convex  with $Q_\nu\in\R^{n_\nu\times n_\nu}$ symmetric positive definite, $c_\nu\in\R^{n_\nu}$, and $a\in\R^m_+$.
The multistrategy vector of all players is denoted  by $x=(x_\nu)_{\nu=1}^N\in\R^n$ and  the \textcolor{black}{rival's} strategies to player $\nu$ by $x_{-\nu}=(x_i)_{i=1,i\neq\nu}^N\in\R^{n-n_\nu}$.
\textcolor{black}{The follower's strategy \textcolor{black}{$y=y(x)$} in \eqref{leader} is a best response to the leader strategies and couples implicitly to the leaders' problems.
}

For simple notation, the model is introduced as a Multi-Leader-(Single-)Follower game.
However, this can be extended to multiple followers:
\begin{remark}(Multiple Followers)
	For $j=1,\dots,N_F$ let the $j$--th follower's optimization problem be
	\begin{equation*}
	\min\limits_{y_j\in \R^{m_j}} \frac{1}{2}y_j^\top Q^j_yy_j-b_j(x)^\top \textcolor{black}{y_j(x)} \quad \mathrm{s.t.}\quad y_j\geq l_j(x),
	\end{equation*}
	 with a positive definite diagonal matrix $Q^j_y\in\R^{m_j\times m_j}$   and $b_j,l_j:~\R^n\rightarrow\R^{m_j}$ \textcolor{black}{are} convex in every component and differentiable.
	\textcolor{black}{This structure is called potential game} and it can be equivalently reformulated as a single optimization problem by summing the objectives and concatenating the constraints, c.f. \cite{MONDERER1996}.
\end{remark}

\textcolor{black}{No matter how many followers are modeled, our approach relies on the property that the follower level can be replaced by its optimal conditions.
In addition, we need an explicit response function of the followers, which is given here in particular by the convexity properties of the follower problem and the diagonality of $Q_y$.}
Since the follower's problem has a strictly convex objective and a convex strategy set, 
we state its unique solution in the following lemma:
\begin{lemma}[Follower's Best Response]\label{Th:bestresponse}
	The follower's optimization problem \eqref{follower} has a unique solution $y^*$ for any given leader strategy vector $x\in\R^n$.
	In particular, this best response function is
	\begin{equation}\label{solY2}
			y^*(x)=\max\left\{Q_y^{-1}b(x),l(x)\right\}.
	\end{equation}

\end{lemma}

\begin{proof}
	The objective and the feasible set are convex for any $x\in\R^n$, \textcolor{black}{in particular, the constraints are linear.}
	Therefore, the KKT conditions are necessary and sufficient for a global minimizer.
	Since the objective is strictly convex, the minimizer is unique for all leader strategies $x$.
	
	To derive the structure, we apply the KKT conditions:
	There exist Lagrange multipliers $\lambda\in\R^m$ such that
	\begin{align*}
		0&=Q_yy-b(x)-\lambda,\\
		0&\leq \lambda~\bot~y-l(x)\geq 0.
	\end{align*}
	Combining these expressions yields the complementarity expression
	$ 0\leq \lambda=Q_yy-b(x) ~\bot~ y-l(x)\geq0$.
	We assumed $Q_y$ to be positive definite and diagonal; therefore, we can write equivalently as
$ 0\leq y-Q^{-1}_yb(x) ~\bot~ y-l(x)\geq0$ and
	$0=\min\left\{ y-Q^{-1}_yb(x) ,y-l(x) \right\}$.
	We obtain \eqref{solY2} by extracting $y$ and changing $\min$ to $\max$.
\end{proof}
\noindent We remark that the $\min$ and $\max$ operator are applied  componentwise to a vector.
If the data $b,l$ are chosen to be componentwise convex, then also the components of best response function $y^*(x)$ are  \textcolor{black}{also} convex in $x$ as a maximum of  convex functions.\\
  
The MLFG  is formulated as a Nash game by plugging the best response \eqref{solY2} in the leader game \eqref{leader} for $\nu=1,\dots,N$, yielding
\begin{equation}\label{NEP22}
\min\limits_{x_\nu} \frac{1}{2} x_\nu^\top Q_\nu x_{\nu}+c_\nu^\top x_\nu+\sum\limits_{i=1}^{m}a_i \max\left\{\left(Q_y^{-1}b(x)\right)_i,l_i(x)\right\} \quad\st\quad x_\nu\in X_\nu.\tag{NEP}
\end{equation}
Each optimization problem has a nonsmooth but convex objective \textcolor{black}{(if $b_i,l_i$ are convex)} and a convex strategy set.
For compact strategy sets, we can prove the existence of Nash equilibria in the following theorem.
\begin{theorem}[Existence of Nash Equilibria for Compact Strategy Sets]\label{Th:MainExistence}
	Assume that the nonsmooth Nash equilibrium problem in \eqref{NEP22} has a  convex and compact joint strategy set $X=X_1\times \dots\times X_N$, where all $X_\nu$ are nonempty.
\textcolor{black}{	Further, assume that $b,l$ are componentwise convex functions on $X$.}
	Then there exists at least one Nash equilibrium.
	
	Therefore, the quadratic multi-leader-follower game given by (\ref{follower}) and (\ref{leader}) has at least one Nash equilibrium.
\end{theorem}
\begin{proof}
	We formulated the MLFG as a convex Nash equilibrium problem (\ref{NEP22}), especially the objectives are continuous in $(x_\nu,x_{-\nu})$ and convex in $x_\nu$ because they are a sum of a strictly convex quadratic term and the maximum of two convex functions.
	Furthermore, we assumed  the admissible strategy sets $X_\nu$ to be nonempty, convex, and compact. 
	Therefore, the conditions of \cite[Theorem 3.1]{nikaido1955note} are fulfilled.
\end{proof}
Alternatively, Theorem \ref{Th:constr} includes a constructive \textcolor{black}{proof of existence}.
\textcolor{black}{In addition to convexity, both existence results rely on compactness of the strategy sets to ensure existence of global minimizer to the  players' optimization problems.}
However, this is clearly not a necessary condition for the existence or uniqueness of Nash equilibria to (\ref{follower},\ref{leader}).

In the remainder of the section, we exchange the compactness assumption with \textcolor{black}{a} coercivity assumption on the best response of the followers exploiting the potentiality of the \textcolor{black}{leaders'} objectives.
Therefore, let the generalized potential function be
\begin{align*}
	\Theta(x)&=\sum\limits_{\nu=1}^{N}\left[\frac{1}{2} x_\nu^\top Q_\nu x_{\nu}+c_\nu^\top x_\nu\right] +\varphi(x),\\
	\varphi(x)&=\sum\limits_{i=1}^{m}a_i \max\left\{\left(Q_y^{-1}b(x)\right)_i,l_i(x)\right\}.
\end{align*}
With this, we can formulate the major assumption:
\begin{assumption}\label{Ass:coerc}
	Assume, \textcolor{black}{there exists} $\rho\geq 0$ \textcolor{black}{and  $\omega_1,\omega_2\in\R$} such that for all $x\in X$ with $\|x\|>\rho$  it holds
	$$\textcolor{black}{\min\left\{0,\varphi(x)\right\}}\geq \omega_1 \|x\| +\omega_2.$$
\end{assumption}
\textcolor{black}{This assumption guarantees that the coercivity of  generalized potential $\Theta$ is preserved.
It includes linear functions $\varphi(x)$ therefore especially settings where $b$ and $l$ are linear.
It excludes polynomials of higher degree which are not bounded from below.} 

\textcolor{black}{
\begin{lemma}[Coercivity of the Generalized Potential $\Theta$]
Let Assumption \ref{Ass:coerc} hold, then the generalized potential function $\Theta$ is coercive.
\end{lemma}
\begin{proof}
We introduce a short hand for the block diagonal matrix $Q=\mathrm{diag}(Q_1,\dots,Q_N)$ and the concatenation of $c_1,\dots,c_N$ as $c$. 
Then we have
\begin{align*}
  \Theta(x)=x^\top Q x+c^\top x +\varphi(x)\geq \mu \|x\|^2 -\|c\|\,\|x\|+\min\left\{0,\varphi(x)\right\},
\end{align*}
where $\mu$ is the smallest eigenvalue of $Q$, recall that $\mu>0$ since all $Q_\nu$ are positive definite.
We rewrite the the right hand side with Assumption \ref{Ass:coerc} and get
\begin{align*}
\Theta(x)\geq \mu \|x\|^2 -\|c\|\,\|x\|+\omega_1\|x\|+\omega_2,
\end{align*}
As $\|x\|\rightarrow\infty$, the right hand side is dominated by a quadratic term with $\mu \|x\|^2\rightarrow\infty$ and we conclude
$$\lim\limits_{\|x\|\rightarrow\infty}\Theta(x)=\infty,$$
i.e. the generalized potential function is coercive.
\end{proof}}
The Assumption \ref{Ass:coerc} and the concluded coercivity property allows us to formulate an existence result without requiring compactness, \textcolor{black}{similarly to the standard result of Monderer and Shapley in \cite{MONDERER1996}}.

\begin{theorem}[Existence  of Nash Equilibria for Coercive Generalized Potential]\label{Th:coerc}
	Let  Assumption \ref{Ass:coerc} hold and   $x^*$ be a global minimizer of the generalized potential $\Theta$, then $x^*$ is also a Nash equilibrium of the MLFG (\ref{follower}-\ref{leader}).
\end{theorem}
\begin{proof}
	Since  Assumption \ref{Ass:coerc} holds, the generalized potential $\Theta$ admits a global minimizer on any closed joint strategy set $X$.
	We demonstrate that $x^*$ is a Nash equilibrium by contradiction.
	Therefore, assume the converse is true, i.e. \textcolor{black}{there exists} a leader $\hat{\nu}$ and a strategy $\hat{x}_{\hat{\nu}}\in X_{\hat{\nu}}$ such that:
	$$\theta_{\hat{\nu}}(\hat{x}_{\hat{\nu}},x^*_{-\hat{\nu}})< \theta_{\hat{\nu}}(x^*_{\hat{\nu}},x^*_{-\hat{\nu}}),$$
or more detailed,
$$ \frac{1}{2}\hat{x}_{\hat{\nu}}^\top Q_{\hat{\nu}}\hat{x}_{\hat{\nu}}+c_{\hat{\nu}}^\top \hat{x}_{\hat{\nu}}+\varphi(\hat{x}_{\hat{\nu}},{x}^*_{-\hat{\nu}})<\frac{1}{2}x^{*\top}_{\hat{\nu}} Q_{\hat{\nu}}{x}^*_{\hat{\nu}}+c_{\hat{\nu}}^\top {x}^*_{\hat{\nu}}+\varphi({x}^*_{\hat{\nu}},{x}^*_{-\hat{\nu}}).
$$
With this, we can overestimate the following expression
\begin{align*}
	\Theta(\hat{x}_{\hat{\nu}},{x}^*_{-\hat{\nu}})=& \sum\limits_{\nu=1,\nu\neq \hat{\nu}}^{N}\left[\frac{1}{2} x^{*\top}_\nu Q_\nu x^*_{\nu}+c_\nu^\top x^*_\nu\right] + \frac{1}{2}\hat{x}_{\hat{\nu}}^\top Q_{\hat{\nu}}\hat{x}_{\hat{\nu}}+c_{\hat{\nu}}^\top \hat{x}_{\hat{\nu}}+\varphi(\hat{x}_{\hat{\nu}},{x}^*_{-\hat{\nu}})\\
	<&  \sum\limits_{\nu=1}^{N}\left[\frac{1}{2} x^{*\top}_\nu Q_\nu x^*_{\nu}+c_\nu^\top x^*_\nu\right] + \varphi({x}^*_{\hat{\nu}},{x}^*_{-\hat{\nu}})=\Theta({x}^*_{\hat{\nu}},{x}^*_{-\hat{\nu}}),
\end{align*}
which is a \textcolor{black}{contradiction to} the optimality of $x^*$.
\end{proof}
In the remainder of this article, we assume  $b$ and $l$ to be linear:
\begin{assumption}\label{Ass:Lin}
	\textcolor{black}{The data is assumed to be linear:
		\begin{equation*}
		b(x)=B^\top x \quad\mbox{and}\quad l(x)=L^\top x,
		\end{equation*}
		where $B,L\in\R^{n\times m}$,  $B_{:,i},L_{:,i}\in\R^{n}$ denotes the $i$-th column, and $B_{\nu,:},L_{\nu,:}$ denote their submatrices of the rows referring to $x_\nu$.}
\end{assumption}
\textcolor{black}{This linearity combined with Assumption \ref{Ass:coerc} directly yields existence of Nash equilibria.
\begin{lemma}[Existence  of Nash Equilibria for Linear Data $b,l$]
	If \textcolor{black}{Assumptions} \ref{Ass:coerc}  and \ref{Ass:Lin} hold,  then \textcolor{black}{there exists} a Nash equilibrium to the MLFG.
	\end{lemma}
\begin{proof}
	With linear data $b,l$, $\varphi$ gets: 
	$$\varphi(x)=\sum\limits_{i=1}^{m}a_i \max\left\{\left(Q_y^{-1}B^\top x\right)_i,\left(L^\top x\right)_i\right\},$$
	which is a linear combination of the maxima of two linear functions.
	Thus, it can be underestimated by an other linear function.
	Therefore, Assumption \ref{Ass:coerc} holds and Theorem \ref{Th:coerc} applies.
\end{proof}
}
\textcolor{black}{In \cite{Steffensen2019P} there is a counterexample for the structure of  MLFG in (\ref{follower},\ref{leader}) with linear data $b,l$.
The data and $X$ are such that the non-differentiable parts of the objectives of the NEP reformulation lie either on the boundary of $X$ or outside the feasible domain.
Or with other words, the $\max$ operator can be uniquely evaluated  on the interior of the feasible set.
In particular, existence and uniqueness of Nash equilibria is derived by a formulation as variational inequality.
In case of diagonal $Q_\nu$, the Nash equilibrium can be explicitly computed in the presented setting.
}

\section{Existence and Uniqueness of Nash Equilibria of Smoothed MLFG} \label{Sec:3}
We formulated the MLFG as a convex but nonsmooth Nash game and proved existence of  equilibria in case of compact strategy sets.
In this section,  we  relax the nonsmoothness of the follower's best response.
For the resulting smooth convex Nash equilibrium problem, we show existence and uniqueness for more general strategy sets.

Similarly to Lemma \ref{Th:bestresponse}, we formulate the follower's KKT conditions with the linear data:
\begin{equation}\label{Eq:KKT}
0\leq y-Q^{-1}_yB^\top x ~\bot~ y-L^\top x\geq0,
\end{equation}
where we replace the complementarity expression by a formulation with a nonlinear complementarity  (NCP) function.
We consider smooth NCP functions of the following type with smoothing parameter $\varepsilon>0$:
\begin{equation}\label{NCPsmooth}
\phi_\varepsilon(\alpha,\beta)=\alpha+\beta-\tilde{\phi_\varepsilon}(\alpha-\beta).
\end{equation}
\textcolor{black}{An example for a smooth NCP function with convex $\tilde{\phi}_\varepsilon$ is  $\phi_\varepsilon^{p}(\alpha,\beta)=\alpha+\beta-\sqrt[p]{(\alpha-\beta)^p+(2\varepsilon)^p}$ for an even number $p\in\mathbb{N}$, this coincides for $p=2$ with the smooth minimum function.
This structure \textcolor{black}{captures also} NCP functions related to the relaxation scheme for mathematical programs with equilibrium constraints (MPEC) introduced by Steffensen and Ulbrich \cite{Sonja10}.
Besides \textcolor{black}{differentiability}  in $\alpha$ and $\beta$, we require the smooth NCP function $\phi_\varepsilon$ to be continuous in the smoothing parameter $\varepsilon$.}

If we apply (\ref{NCPsmooth}) on the KKT system \eqref{Eq:KKT}, we obtain (component wise) for $\varepsilon>0$ the expression
\begin{align}
0&= y - Q^{-1}_yB^\top x + y-L^\top x -\tilde{\phi_\varepsilon}\left(y-Q^{-1}_yB^\top x - y+L^\top x\right).
\end{align}
Therefore, we write the best response function as:
\begin{equation}\label{Eq:FollowerSolution}
y_\varepsilon(x)=\frac{1}{2}\left[\left(L^\top +Q^{-1}_yB^\top\right) x +\tilde{\phi_\varepsilon}\left(\left(L^\top-Q^{-1}_yB^\top\right) x\right)\right].
\end{equation}

The smoothed best response function $y_\varepsilon$ plugged  into the leader's objectives yields  a smooth Nash equilibrium problem with positive smoothing parameter $\varepsilon$; for $\nu=1,\dots, N$, we have
\begin{align}\label{Eq:SmoothNEP} \tag{\mbox{NEP$(\varepsilon) $}}
\begin{split}
\min\limits_{x_\nu\in\R^{n_\nu}}  \theta^\varepsilon_\nu(x_\nu,x_{-\nu})=&\frac{1}{2} x_\nu^\top Q_\nu x_{\nu}+c_\nu^\top x_\nu\\ &+ \frac{1}{2}\sum\limits_{i=1}^m a_i \left[\left(L^\top +Q^{-1}_yB^\top\right) x +\tilde{\phi_\varepsilon}\left(\left(L^\top -Q^{-1}_yB^\top\right) x\right)\right]_i \\
\st ~ x_\nu\in X_\nu. \quad\quad&
\end{split}
\end{align}
For this game, we can state an existence and uniqueness theorem.

\begin{theorem}[Existence and Uniqueness]
	Assume that the  Nash equilibrium problem  \ref{Eq:SmoothNEP} has a  convex and closed strategy set $X=X_1\times\dots \times X_N$, where all $X_\nu$ are nonempty, and $\tilde{\phi_\varepsilon}$ is convex and smooth. 
	Then the  Nash equilibrium problem has a unique  equilibrium for every smoothing parameter $\varepsilon>0$.
\end{theorem}
\begin{proof}
	Due to \cite[Proposition 1.4.2]{facchinei2007finite} and the convexity assumptions, a strategy $x\in X$ is a Nash equilibrium if and only if $x$ solves the variational inequality $\mathrm{VI}(X,{\theta'}^\varepsilon)$, where
	$${\theta'}^\varepsilon(x)=\begin{pmatrix} \nabla_{x_1} \theta^\varepsilon_1 (x_1,x_{-1})\\ \vdots \\\nabla_{x_N} \theta^\varepsilon_N (x_N,x_{-N})
	\end{pmatrix}. $$
	Since we assumed a convex and closed strategy set $X$, the variational inequality has a unique solution if ${\theta'}^\varepsilon$ is uniformly monotone \cite[Theorem 2.3.3]{facchinei2007finite}.
	The remainder of the proof demonstrates this.
	
	We introduce a short hand for the linear term $A=L^\top -Q_y^{-1}B^\top $, the block diagonal matrix $Q=\mathrm{diag}(Q_1,\dots,Q_N)$, and the concatenation of $c_1,\dots,c_N$ as $c$. 
	Let $x,\hat{x}\in X$ be distinct, then we have
	\begin{align*}
		&(x-\hat{x})^\top \left({\theta'}^\varepsilon(x)-{\theta'}^\varepsilon(\hat{x})\right), \\
		=&(x-\hat{x})^\top  \Bigg[Qx+c + \frac{1}{2}(L^\top +Q_y^{-1}B^\top )^\top a+\frac{1}{2}\sum\limits_{i=1}^m a_i A_{i,:}^\top  \tilde{\phi}_\varepsilon'((Ax)_i)  \\
		 &   -\left(Q\hat{x} +c+\frac{1}{2}(L^\top +Q_y^{-1}B^\top )^\top a+\frac{1}{2}\sum\limits_{i=1}^m a_i A_{i,:}^\top  \tilde{\phi}_\varepsilon'((A\hat{x})_i)\right)\Bigg], \\
		  = & (x-\hat{x})^\top   \left[ Q(x-\hat{x}) + \frac{1}{2}\sum\limits_{i=1}^m a_i A_{i,:}^\top  \left(\tilde{\phi}_\varepsilon'((Ax)_i)-\tilde{\phi}_\varepsilon'((A\hat{x})_i)\right) \right].
	\end{align*}
	We apply the mean value theorem for $\tilde{\phi}_\varepsilon'$, thus there exists a $t\in[0,1]$ such that we derive from the last equation with $\zeta_i=t(Ax)_i+(1-t)(A\hat{x})_i$ a lower \textcolor{black}{bound}:
	\begin{align*}
		&(x-\hat{x})^\top \left({\theta'}^\varepsilon(x)-{\theta'}^\varepsilon(\hat{x})\right),\\
		 =&  (x-\hat{x})^\top    Q(x-\hat{x}) + \frac{1}{2}\sum\limits_{i=1}^m a_i  \underbrace{\tilde{\phi}_\varepsilon''(\zeta_i)}_{\geq 0} \underbrace{(x-\hat{x})^\top  A_{i,:}^\top  A_{i,:} (x-\hat{x})}_{\geq 0}, \\
		\geq
		& (x-\hat{x})^\top    Q(x-\hat{x}) \geq \mu\|x-\hat{x}\|^2,
	\end{align*}
	where $\mu$ is the  smallest eigenvalue of  $Q$. It is positive because $Q$ is symmetric positive definite.
Thus $ {\theta'}^\varepsilon$ is uniformly monotone.
	Therefore, the variational inequality $\mathrm{VI}(X,{\theta'}^\varepsilon)$ has exactly one solution which is the unique Nash equilibrium of the smoothed game  \ref{Eq:SmoothNEP}.
	\end{proof}

\textcolor{black}{The remainder of this section is dedicated to the characterization of the unique Nash equilibrium.}
Therefore, we recall the assumptions made for the  smooth Nash equilibrium problem \ref{Eq:SmoothNEP} for positive smoothing parameter $\varepsilon$:
\begin{assumption}\label{Assump:SNEP}
	We assume that the data of the MLFG and its smooth Nash game reformulation satisfy  for $\nu=1,\dots, N$ the following properties.
	\begin{itemize}
		\item $Q_\nu\in\R^{n_\nu\times n_\nu}$ symmetric positive definite,
		\item $c_\nu\in\R^{n_\nu}$,
		\item $a\in\R_+^m$,
		\item $X_\nu=\left\{ x_\nu\in\R^{n_\nu} | g_\nu(x_\nu)\leq 0 \right\}\subseteq \R^{n_\nu}$ nonempty, convex, and closed;
		\item $ g_\nu: \R^{n_\nu}\rightarrow\R^{m_\nu} $  at least twice differentiable, and convex;
		\item $Q_y\in\R^{m\times m}$  positive definite and diagonal,
		\item $b(x)=B^\top x$, $l(x)=L^\top x$, with $B,L \in \R^{n\times m}$,
		\item smooth NCP function of the form $ \phi_\varepsilon(\alpha,\beta)=\alpha+\beta-\tilde{\phi_\varepsilon}(\alpha-\beta)$ where $\tilde{\phi_\varepsilon}$  is at least twice differentiable and convex for every $\varepsilon>0$.
	\end{itemize}
\end{assumption}

\textcolor{black}{Based on these assumptions, we characterize the Nash equilibrium of NEP($\varepsilon$) by its KKT conditions in the following lemma:
\begin{lemma}[KKT of NEP($\varepsilon$)]\label{lem:KKTNEPeps}
	Given \textcolor{black}{Assumption} \ref{Assump:SNEP} and some constraint qualification for $X_\nu$, the KKT conditions \textcolor{black}{of the leader-level problems} are necessary and sufficient for the global minimizer of each leader problem in \ref{Eq:SmoothNEP}.
	In particular, we have the KKT conditions of player $\nu$'s optimization problem for $\nu=1,\dots N$:
	\begin{subequations}\label{Eq:KKTepsilon}
		\begin{align}
		\begin{split}
		0=&Q_\nu x_\nu+c_\nu +\frac{1}{2} (L^\top +Q_y^{-1}B^\top )^\top _{\nu,:} a \\
		&+ \frac{1}{2}\sum\limits^{m}_{i=1}a_i(L^\top -Q_y^{-1}B^\top )^\top _{\nu,i} ~\tilde{\phi_\varepsilon}'\left([(L^\top -Q_y^{-1}B^\top ) x]_i\right) + \nabla_{x_\nu}g_\nu (x_\nu) \lambda_\nu,
		\end{split}\\
		0&=\min\left\{ \lambda_\nu , -g_{\nu}(x_\nu)\right\},
		\end{align}
	\end{subequations}
with the Lagrange multiplier $\lambda_\nu\in\R_+^{m_\nu}$ and the Jacobian of the constraints\\
 $	 \nabla_{x_\nu}g_\nu (x_\nu)= (\nabla_{x_\nu} g_{\nu_1}(x_\nu),\dots,\nabla_{x_\nu} g_{\nu_{m_\nu}}(x_\nu) )\in\R^{n_\nu\times m_\nu}$.
\end{lemma}
\begin{proof}
	KKT is necessary and sufficient for every leader problem in presence of a constraint qualification, since the objectives are strictly convex and the  strategy set are convex and closed.
	Therefore, the joint KKT system is {necessary and sufficient}  for the unique Nash equilibrium.
	The follower's solution can be explicitly computed by the leader's solutions, c.f. (\ref{Eq:FollowerSolution}).
\end{proof}
The imposed constraint qualification is specified in the following section.}
\section{Relation between MLFG/NEP and NEP($\varepsilon$)}\label{Sec:4}
\textcolor{black}{In this section, we study the relationship between the original MLFG (or its equivalent NEP) and the smoothed problem NEP($\varepsilon$).
Our aim is to demonstrate that Nash equilibria of NEP($\varepsilon$) \textcolor{black}{converge} to a Nash equilibrium of the original problem as the smoothing parameter $\varepsilon\rightarrow0$.
}

\textcolor{black}{Instead of looking at optimality conditions for the nonsmooth formulation, we are considering  stationarity concepts for mathematical programs with complementarity constraints (MPCC) next.}
\begin{definition}[Mathematical Programs with Complementarity Constraints]
	Let $f:~\R^n\rightarrow\R$, $g:~\R^n\rightarrow\R^{m}$, and $G_1,G_2:~\R^n\rightarrow\R^{l}$ be smooth functions.
	Then we call:
	\begin{align}\begin{split}\label{MPCC}
	\min\limits_z~& f(z)\\
	\mathrm{s.t.}~ &g(z)\leq0\\
	& 0=\min\left\{ G_1(z), G_2(z)  \right\} 
	\end{split}\tag{MPCC}
	\end{align}
	a mathematical program with complementarity constraints.	
\end{definition}
In fact, each leader problem \eqref{leader} can be formulated as an MPCC:
\begin{align}\begin{split}\label{Eq:MPCCnu}
\min\limits_{x_\nu,y}~&\theta_\nu(x_\nu,x_{-\nu})=\frac{1}{2}x_\nu^\top Q_\nu x_\nu+c_\nu^\top x_\nu+ a^\top \textcolor{black}{y(x)}\\
\mathrm{s.t.}~&g_\nu(x_\nu)\leq 0\\
& 0=\min\{G_1(x_\nu,x_{-\nu},\textcolor{black}{y(x)}),G_2(x_\nu,x_{-\nu},\textcolor{black}{y(x)})\},
	\end{split}\tag{MPCC$_\nu$}
\end{align}
with the complementarity constraints $G_1(x_\nu,x_{-\nu},\textcolor{black}{y(x)})=\textcolor{black}{y(x)}-(Q_y^{-1}B^\top)^\top x$ and $G_2(x_\nu,x_{-\nu},\textcolor{black}{y(x)})=\textcolor{black}{y(x)}-L^\top x$.
For $\nu=1,\dots, N$, the \eqref{Eq:MPCCnu} form together the generalized Nash equilibrium problem (GNEP) formulation of the MLFG, where the complementarity constraints are shablack constraints.

Similar to KKT points, there is a variety of stationary concepts for MPCC, we introduce the strongest one here: strongly stationary points.
We adapt \cite[Theorem 2]{ScheelScholtes00}.
\begin{definition}[S-Stationarity]\label{Def:SStation}
   We call $\bar{z}$ a strongly \textcolor{black}{(S-)stationary} point of \eqref{MPCC} if there exist multipliers $(\lambda, \Gamma_1, \Gamma_2)\in\R^{m+l+l}$ with:
	\begin{align}	\begin{split}\label{Eq:Sstationarity}
	0=\nabla_z f(\bar{z}) + \sum\limits_{i=1}^{m} \lambda_i \nabla_z g_i(\bar{z})& - \sum\limits_{i=1}^{l} \Gamma_{1,i} \nabla_z G_{1,i}(\bar{z}) - \sum\limits_{i=1}^{l} \Gamma_{2,i} \nabla_z G_{2,i}(\bar{z}),\\
	g(\bar{z})&\leq 0,\\
	\lambda & \geq 0,\\
	g_i(\bar{z})\lambda_i &= 0, \quad i=1,\dots, m,\\
	\min\left\{ G_{1,i}(\bar{z}), G_{2,i}(\bar{z})  \right\} &=0,  \quad i=1,\dots,l,\\
	G_{1,i}(\bar{z})\Gamma_{1,i}&=0,  \quad i=1,\dots,l,\\
	G_{2,i}(\bar{z})\Gamma_{2,i}&=0,\quad i=1,\dots,l,\\
	\Gamma_{1,i},\Gamma_{2,i}&\geq 0,\quad i: G_{1,i}(\bar{z})=G_{2,i}(\bar{z})=0.
	\end{split}	\end{align}
\end{definition}

\textcolor{black}{Besides suitable stationarity concepts, MPCCs also require proper constraint qualifications. 
Here, the Constant Rank Constraint Qualification (CRCQ) for MPCC is defined  similarly to its version for MPEC in \cite[Def. 2.2]{Sonja10}.
Originally, the CRCQ was introduced for NLP by \cite{Janin1984}.
\begin{definition}(MPCC-CRCQ)
 The constraint qualification MPCC-CRCQ holds in the feasible point $\bar{z}$ of  \eqref{MPCC}, if  for every $\mathcal{K}_g$, $\mathcal{K}_1$, and $\mathcal{K}_2$ with
	\begin{align*}
		\mathcal{K}_g &\subseteq I_g(\bar{z})=\left\{i \in \{1,\dots,m\}| g_i(\bar{z})=0 \right\},\\
		\mathcal{K}_1 &\subseteq I_1(\bar{z})=\left\{i \in \{1,\dots,l\}| G_{1,i}(\bar{z})=0 \right\},\\
		\mathcal{K}_2 &\subseteq I_2(\bar{z})=\left\{i \in \{1,\dots,l\}| G_{2,i}(\bar{z})=0 \right\},\\
	\end{align*}
	there exists a neighborhood $\mathcal{U}(\bar{z})$ such that for every ${z}\in \mathcal{U}(\hat{z})$ the family of gradient vectors
	\begin{equation*}
        \{ \nabla g_i({z})|i\in\mathcal{K}_g  \}\cup  \{ \nabla G_{1,i}({z})|i\in\mathcal{K}_1  \} \cup \{ \nabla G_{2,i}({z})|i\in\mathcal{K}_2  \},
	\end{equation*}
	has the same rank as the family
	\begin{equation*}
	\{ \nabla g_i(\bar{z})|i\in\mathcal{K}_g  \}\cup  \{ \nabla G_{1,i}(\bar{z})|i\in\mathcal{K}_1  \} \cup \{ \nabla G_{2,i}(\bar{z})|i\in\mathcal{K}_2  \}.
	\end{equation*}
\end{definition}
}

In the following theorem we show that a limit of the Nash equilibria of the smoothed problems is indeed strongly stationary with respect to the leader problems.

\begin{theorem}[S-stationarity and Convergence of Multipliers]
\textcolor{black}{	Let $(\varepsilon_k)_k$ be a positive sequence with $\varepsilon_k\rightarrow 0$ and let $(x^*(\varepsilon_k))_k$ be the associated sequence of the unique  Nash equilibria to \emph{NEP($\varepsilon_k$)}.
	Assume this sequence  converges with $x^*(\varepsilon_k)\rightarrow x^*(0)$ and every subvector of the limit $x_\nu^*(0)$ satisfies MPCC-CRCQ for its leader's  \emph{MPCC}$_\nu$.
}

	\textcolor{black}{Then, \textcolor{black}{there exists} a bounded sequence of multipliers associated to  $(x^*(\varepsilon_k))_k$ and the limit $x_\nu^*(0)$ is a strongly stationary point for \eqref{Eq:MPCCnu} for $\nu=1,\dots,N$.}
\end{theorem}
\begin{proof}
	\textcolor{black}{The proof consists of three major parts: First, we introduce notation and modest conclusions.
	Second, we  demonstrate that there exists a bounded sequence of multipliers associated to $(x^*_{\nu}(\varepsilon_k))_k$.
	Third, we verify the conditions of strong stationarity by constructing suitable multipliers.}\\
	
 \textcolor{black}{(i)
 In order to keep the notation simple, let $x_\nu^k=x_\nu^*(\varepsilon_k)$ and the limit $\bar{x}_\nu=x_\nu^*(0)$ and respectively the concatenations $x^k$ and $\bar{x}$, both without the player index.
	Further we introduce the short forms for $k\in\mathbb{N}$ and $i=1,\dots, m$:
	\begin{equation*}
		z_i^k=\left( (L^\top-Q_y^{-1}B^\top)x^k\right)_i,  \bar{z}_i=\left( (L^\top-Q_y^{-1}B^\top)\bar{x}\right)_i, \text{ and } \xi^k_i=\tilde{\phi}'_{\varepsilon_k}(z^k_i),
	\end{equation*}
	 and the concatenations are denoted by $z^k$, $\bar{z}$, and $\xi^k$, respectively.
	 Recall, it holds that $z_i^k\rightarrow\bar{z}_i$ and  $|\xi_i^k|\leq1$ for  $i=1,\dots, m$ and all $k\in\mathbb{N}$.}
	
	\textcolor{black}{Due to stationarity of  $x^k_{\nu}$ for all $\nu$ and $k\in\mathbb{N}$, we  use  Lemma \ref{lem:KKTNEPeps}, i.e. \eqref{Eq:KKTepsilon}, and conclude that there exists multipliers $\lambda_\nu^k\in \R^{m_\nu}$ such that:
\begin{subequations}
\begin{align}
	0=&Q_\nu x^k_\nu+c_\nu +\frac{1}{2} (L^\top +Q_y^{-1}B^\top )^\top _{\nu,:} a + \frac{1}{2}\sum\limits^{m}_{i=1}a_i(L^\top -Q_y^{-1}B^\top )^\top _{\nu,i} ~\xi_i^k + \nabla_{x_\nu}g_\nu (x^k_\nu) \lambda_\nu^k,\label{KKT10a}\\
	0=&\min\left\{ \lambda^k_\nu , -g_{\nu}(x^k_\nu)\right\},\label{eq:KKTproof0}
	\end{align}
	\end{subequations}
	Hence, by rearranging the sum in \eqref{KKT10a} we have:
	\begin{equation}\label{eq:KKTproof}
			0=Q_\nu x^k_\nu+c_\nu + \sum\limits_{i=1}^{m}L^\top_{\nu,i}\frac{a_i}{2}(1+\xi_i^k)+   \sum\limits_{i=1}^{m}Q_y^{-1}B^\top_{\nu,i}\frac{a_i}{2}(1-\xi_i^k)+ \sum\limits_{j=1}^{m_\nu}\nabla_{x_\nu}g_{\nu,j} (x^k_\nu) \lambda_{\nu,j}^k,
	\end{equation}
Next, we construct a multiplier vector $\hat{\lambda}^k_\nu$ for all $\nu$ and sufficiently large $k$ such that the following inclusion holds:
$$ \mathrm{supp}\left(\hat{\lambda}^k_\nu\right)\subseteq  \mathrm{supp}\left({\lambda}^k_\nu\right)\subseteq I_{g_\nu}(\bar{x}_\nu)=\left\{j \left|g_{\nu,j}(\bar{x}_\nu)=0 \right.\right\},$$	
where $\mathrm{supp}(\cdot)$ denotes the support of a vector, i.e. the set of indices of its nonzero entries, the multiplier $\hat{\lambda}^k_\nu$  shall satisfy \eqref{eq:KKTproof} and the family of vectors
\begin{equation}\label{Eq:linindep}
	\left\{  \nabla_{x_\nu}g_{\nu,j} (x^k_\nu) \left| j\in \mathrm{supp}\left(\hat{\lambda}^k_\nu\right)    \right. \right\}
\end{equation}
is linearly independent.}\\

\textcolor{black}{(ii) 
Now, we demonstrate  by contradiction that the multiplier sequence $\left(\hat{\lambda}^k_\nu\right)_k$ is bounded for all $\nu$.
Therefore, it is assumed that the converse is true, i.e. there exists $\hat{\nu}$ such that $\left(\hat{\lambda}^k_{\hat{\nu}}\right)_k$ is unbounded.
By assumption, we have $\|\hat{\lambda}^k_{\hat{\nu}}\|\rightarrow \infty$ as $k\rightarrow\infty$ and thus  it holds, that $\|\hat{\lambda}^k_{\hat{\nu}}\|>0$ for sufficiently large $k$.
Therefore, we can define the auxiliary sequence
\begin{equation}\label{Eq:auxsequence}
	 \tilde{\lambda}^k_{\hat{\nu}}=\frac{\hat{\lambda}^k_{\hat{\nu}}}{\|\hat{\lambda}^k_{\hat{\nu}}\|}.
\end{equation}
Clearly, all elements of this sequence are normalized, i.e. $\|\tilde{\lambda}^k_{\hat{\nu}}\|=1$.
This in turn implies that the sequence is bounded and admits a convergent subsequence $\left(\tilde{\lambda}^l_{\hat{\nu}}\right)_{l\in K}\rightarrow\tilde{\lambda}_{\hat{\nu}}$ with $K \subseteq\mathbb{N}$.
It follows, that also the limit of the subsequence satisfies $\|\tilde{\lambda}_{\hat{\nu}}\|=1$ and there exists an index $j_0\in\{1,\dots,m_\nu\}$ with $\tilde{\lambda}_{\hat{\nu},j_0}>0$ such that $\mathrm{supp}\left(\tilde{\lambda}_{\hat{\nu}}\right)\neq\emptyset$.
Now, divide \eqref{eq:KKTproof} by norm of the unbounded sequence $\|\hat{\lambda}^l_{\hat{\nu}}\|$ and take the limit $l\rightarrow\infty$.
Note, that the first four terms of \eqref{eq:KKTproof} are bounded and thus they vanish in the limit.
It remains to compute the limit of the last term:
$$0=\lim\limits_{l\rightarrow \infty, l\in K}\frac{1}{\|\hat{\lambda}^l_{\hat{\nu}}\|}\sum\limits_{j\in\mathrm{supp}\left( \hat{\lambda}_{\hat{\nu}}^k \right)} \nabla_{x_{\hat{\nu}}}g_{{\hat{\nu}},j} (\bar{x}_{\hat{\nu}}) \hat{\lambda}_{\hat{\nu},j}^k,$$
with \eqref{Eq:auxsequence} and (for sufficiently large $k$) $\mathrm{supp}\left(\tilde{\lambda}_{\hat{\nu}}   \right)\subseteq\mathrm{supp}\left(\tilde{\lambda}_{\hat{\nu}}^k   \right)=\mathrm{supp}\left(\hat{\lambda}_{\hat{\nu}}^k   \right)$, we get 
$$0=\sum\limits_{j\in\mathrm{supp}\left( \tilde{\lambda}_{\hat{\nu}} \right)} \nabla_{x_{\hat{\nu}}}g_{{\hat{\nu}},j} (\bar{x}_{\hat{\nu}}) \tilde{\lambda}_{{\hat{\nu}},j}.$$
This implies with $\mathrm{supp}\left(\tilde{\lambda}_{\hat{\nu}}\right)\neq\emptyset$, that the set of vectors
\begin{equation*}
\left\{ \nabla_{x_{\hat{\nu}}}g_{{\hat{\nu}},j} (\bar{x}_{\hat{\nu}}) \left| j\in \mathrm{supp}\left(\tilde{\lambda}_{\hat{\nu}}\right)    \right. \right\},	
\end{equation*}
is  linearly dependent.
Thus, by MPCC-CRCQ and $\mathrm{supp}\left(\tilde{\lambda}_{\hat{\nu}}   \right)\subseteq\mathrm{supp}\left(\hat{\lambda}_{\hat{\nu}}^k   \right)$, also 
\begin{equation*}
\left\{ \nabla_{x_{\hat{\nu}}}g_{{\hat{\nu}},j} ({x}^k_{\hat{\nu}}) \left| j\in \mathrm{supp}\left(\hat{\lambda}^k_{\hat{\nu}}\right)    \right. \right\},	
\end{equation*}
is  linearly dependent, which is a contradiction to the linear independence of \eqref{Eq:linindep} for sufficiently large $k$.
Therefore the assumption $\left(\hat{\lambda}^k_{\hat{\nu}}\right)_k$ being unbounded, because \eqref{eq:KKTproof} could not be satisfied by such a multiplier sequence.
Therefore, the sequence of multipliers  associated to $(x^k)_k$ is bounded, i.e. it exist a multiplier vector $\bar{\lambda}$ associated to $\bar{x}$.}\\ 

\textcolor{black}{(iii) In the following, we verify the conditions of \textcolor{black}{S-stationarity} of Definition \ref{Def:SStation}, i.e. (\ref{Eq:Sstationarity}) exemplary for on leader $\nu$, and begin with applying them to \eqref{Eq:MPCCnu}:
\begin{subequations}\label{Eq:SstationarityProof}
	\begin{align}
	0=\begin{pmatrix}Q_\nu x_\nu+c_\nu\\ a\end{pmatrix}	+ \sum\limits_{i=1}^{m} \lambda_i\begin{pmatrix}\nabla_{x_\nu} g_{\nu,i}(x_\nu)\\0\end{pmatrix}  &- \sum\limits_{i=1}^{l} \Gamma_{1,i}\begin{pmatrix}-(Q_y^{-1}B^\top)_{\nu,i}\\e_i\end{pmatrix} - \sum\limits_{i=1}^{l} \Gamma_{2,i}\begin{pmatrix}
	-L_{\nu,i}\\e_i
	\end{pmatrix},\label{Eq:SstationarityProofa}\\
	g_\nu(x_\nu)&\leq 0\label{Eq:SstationarityProofb},\\
	\lambda_\nu & \geq 0\label{Eq:SstationarityProofc},\\
	g_{\nu,i}(x_\nu)\lambda_{\nu,i} &= 0,\quad i=1,\dots, m_\nu\label{Eq:SstationarityProofd},\\
	\min\left\{	(y-(Q_y^{-1}B^\top)^\top x)_i,(y-L^\top x)_i\right\}&=0,  \quad i=1,\dots,m, \label{Eq:SstationarityProofe}\\
	(y-(Q_y^{-1}B^\top)^\top x)_i\Gamma_{1,i}&=0,\quad  i=1,\dots,m\label{Eq:SstationarityProoff},\\
	(y-L^\top x)_i\Gamma_{2,i}&=0,\quad i=1,\dots,m\label{Eq:SstationarityProofg},\\
	\Gamma_{1,i},\Gamma_{2,i}&\geq 0,\quad i: 	(y-(Q_y^{-1}B^\top)^\top x)_i=(y-L^\top x)_i\label{Eq:SstationarityProofh}.
	\end{align}
\end{subequations}
The remainder of the proof demonstrates that \eqref{Eq:SstationarityProof} is satisfied for the limit strategy $\bar{x}_\nu$ and its multiplier $\bar{\lambda}_\nu$, and we construct the additional multipliers $\Gamma_{1,i},\Gamma_{2,i}$.
We begin with the limit of \eqref{eq:KKTproof} for $k\rightarrow\infty$ in the sense of a suitable subsequence and get the expression:
\begin{equation}\label{Eq:KKTproof2}
			0=Q_\nu \bar{x}_\nu+c_\nu + \sum\limits_{i=1}^{m}L^\top_{\nu,i}\frac{a_i}{2}(1+\bar{\xi}_i)+   \sum\limits_{i=1}^{m}Q_y^{-1}B^\top_{\nu,i}\frac{a_i}{2}(1-\bar{\xi}_i)+ \sum\limits_{j=1}^{m_\nu}\nabla_{x_\nu}g_{\nu,j} (\bar{x}_\nu) \bar{\lambda}_{\nu,j},
\end{equation}
where $\bar{\xi}_i=\lim\limits_{k\rightarrow\infty}\xi_i^k=\lim\limits_{k\rightarrow\infty}\tilde{\phi}'_{\varepsilon_k}(z_k^i)$, note that $\bar{\xi}^i\in\{-1,0,1\}$.
If we choose  the multipliers to the complementarity constraints to be:
\begin{equation*}
\bar{\Gamma}_{1,i}=\frac{a_i}{2}\left(1+\bar{\xi}_i\right),	\bar{\Gamma}_{2,i}=\frac{a_i}{2}\left(1-\bar{\xi}_i\right),
\end{equation*}
and since $a_i-\Gamma_{1,i}-\Gamma_{2,i}=a_i-a_i/2(1+\bar{\xi}_{i})-a_i/2(1-\bar{\xi}_{i})=0$ for all $i=1,\dots,m_\nu$ and \eqref{Eq:KKTproof2}, then the condition in \eqref{Eq:SstationarityProofa} follows.}

\textcolor{black}{Feasibility (\ref{Eq:SstationarityProofb}-\ref{Eq:SstationarityProofd}) is due to continuity of $g_\nu$ and the convergent subsequence of the multipliers, which gives us $\bar{\lambda}$.
The feasibility of the complementarity constraint \eqref{Eq:SstationarityProofe} is due to choice of $\tilde{\phi}_\varepsilon$, which belongs to a smooth NCP function, c.f. (\ref{Eq:KKT}-\ref{Eq:FollowerSolution}).}

\textcolor{black}{\noindent It remains to demonstrate (\ref{Eq:SstationarityProoff}-\ref{Eq:SstationarityProofh}):}

\textcolor{black}{In case $G_{1,i}(\bar{x}_\nu,\bar{x}_{-\nu})> 0$, then by feasibility of $\bar{x}$ we have $G_{2,i}(\bar{x}_\nu,\bar{x}_{-\nu})= 0$ and thus  $\bar{\xi}_i=-1$  such that $\Gamma_{1,i}=0$.
If otherwise $G_{2,i}(\bar{x}_\nu,\bar{x}_{-\nu}) > 0$, then by the  feasibility of $\bar{x}$ it holds that $G_{1,i}(\bar{x}_\nu,\bar{x}_{-\nu})= 0$ and thus  $\bar{\xi}_i=-1$, i.e. $\Gamma_{2,i}=0$.
Since both arguments hold for all $i=1,\dots, m$, this yields (\ref{Eq:SstationarityProoff}-\ref{Eq:SstationarityProofg}).}

\textcolor{black}{Moreover, both multipliers satisfy $\bar{\Gamma}_{1,i}, \bar{\Gamma}_{2,i}\geq 0$ for all $i$ because $a_i>0$ and $|\bar{\xi}_i|\leq 1$; therefore,  \eqref{Eq:SstationarityProofh} holds.}

\textcolor{black}{Hence, the strategy $\bar{x}_\nu$ is strongly stationary for MPCC$_\nu$ and since we derived \eqref{Eq:SstationarityProof} for an arbitrary leader $\nu$, the proof is complete.}

\end{proof}
\textcolor{black}{In the preceding theorem, it is demonstrated that the sequence of multipliers has accumulation points in presence of a suitable constraint qualification if the primal variables converge.
We have seen that the limit of the Nash equilibria of NEP($\varepsilon$) are in fact strongly stationary to the original MLFG.
However, the following theorem goes even further by demonstrating that a limit of these Nash equilibria is in fact a Nash Equilibrium of the MLFG.}
\begin{theorem}[A Nash Equilibrium of the MLFG]\label{Th:constr}
\textcolor{black}{	Let $(\varepsilon_k)_k$ be a positive sequence with $\varepsilon_k\rightarrow 0$ and let $(x^*(\varepsilon_k))_k$ be the associated sequence of the unique  Nash equilibria to \emph{NEP($\varepsilon_k$)}.}
	
\textcolor{black}{	Then any accumulation point of $(x^*(\varepsilon_k))_k$ for a positive sequence $\varepsilon_k\rightarrow 0$ is a Nash equilibrium to \emph{NEP}
	and therefore of the MLFG.}
\end{theorem}
\begin{proof}
\textcolor{black}{Recall, since $x^*(\varepsilon)$ is the unique Nash equilibrium of \ref{Eq:SmoothNEP} for any $\varepsilon>0$, it holds by the definition of a Nash equilibrium that for all $\nu=1,\dots,N$:
\begin{equation*} \label{Eq:DefNE}
	\theta_\nu^\varepsilon(x_\nu^*(\varepsilon),x_{-\nu}^*(\varepsilon))\leq 	\theta_\nu^\varepsilon(x_\nu,x_{-\nu}^*(\varepsilon)) \text{ for all } 	x_\nu\in X_\nu.
\end{equation*} 
We prove that a limit strategy $x^{*}(0)$ is a Nash equilibrium to \ref{NEP22} by contradiction.
Assume, $x^{*}(0)$ is not a Nash equilibrium to \ref{NEP22}, then \textcolor{black}{there exists} $\hat{x}\in X$ such that:
\begin{equation*}
\theta_\nu(x_\nu^*(0),x_{-\nu}^*(0)) > 	\theta_\nu(\hat{x}_\nu,x_{-\nu}^*(0)) \text{ for a leader } 	\nu.
\end{equation*} 
We define the distance as
\begin{equation}\label{Eq:ProofDist}
	\varrho= \theta_\nu(x_\nu^*(0),x_{-\nu}^*(0))- 	\theta_\nu(\hat{x}_\nu,x_{-\nu}^*(0))
\end{equation}
Recall that all objectives $\theta_\nu^\varepsilon$ are continuous in the strategies $x_\nu, x_{-\nu}$ and the smoothing parameter $\varepsilon$, in particular in $\varepsilon=0$.
Then \textcolor{black}{there exists} $\hat{\varepsilon}$ such that for all $\varepsilon\in (0,\hat{\varepsilon})$,   the following relations hold for any $\nu$:
\begin{subequations}\label{Eq:ContinuityProperties}
\begin{align}
	\left|\theta^\varepsilon_\nu(x_\nu^*(\varepsilon),x_{-\nu}^*(\varepsilon))  - 	\theta^\varepsilon_\nu(x_{\nu}^*(0),x_{-\nu}^*(\varepsilon))  \right| & \leq \frac{\varrho}{6} &\text{(continuity in $x_\nu$)}\label{Eq:ContinuityProperties1}\\
	\left|\theta^\varepsilon_\nu(x_\nu^*(0),x_{-\nu}^*(\varepsilon))  - 	\theta^\varepsilon_\nu(x_{\nu}^*(0),x_{-\nu}^*(0))  \right| & \leq \frac{\varrho}{6}&\text{(continuity in $x_{-\nu}$)}\label{Eq:ContinuityProperties2}\\
	\left|\theta^\varepsilon_\nu(x_\nu^*(0),x_{-\nu}^*(0))  - 	\theta_\nu(x_{\nu}^*(0),x_{-\nu}^*(0))  \right| & \leq \frac{\varrho}{6}&\text{(continuity in $\varepsilon$)}\label{Eq:ContinuityProperties3}\\
	\left|\theta^\varepsilon_\nu(\hat{x}_{\nu},x_{-\nu}^*(\varepsilon))  - 	\theta_\nu(\hat{x}_{\nu},x_{-\nu}^*(\varepsilon))  \right| & \leq \frac{\varrho}{6}&\text{(continuity in $\varepsilon$)}\label{Eq:ContinuityProperties4}\\
	\left|\theta_\nu(\hat{x}_{\nu},x_{-\nu}^*(\varepsilon))  - 	\theta_\nu(\hat{x}_{\nu},x_{-\nu}^*(0))  \right| & \leq \frac{\varrho}{6}&\text{(continuity in $x_{-\nu}$)}\label{Eq:ContinuityProperties5}
\end{align}
\end{subequations}
It follows with (\ref{Eq:ContinuityProperties1}-\ref{Eq:ContinuityProperties3}) and triangle inequality, that
\begin{align*}
	\theta_\nu^\varepsilon(x^*_\nu(\varepsilon),x^*_{-\nu}(\varepsilon)) &\geq&	\theta_\nu(x^*_\nu(0),x^*_{-\nu}(0))  ~&-& \left| 	\theta_\nu^\varepsilon(x^*_\nu(\varepsilon),x^*_{-\nu}(\varepsilon)) -\theta_\nu(x^*_\nu(0),x^*_{-\nu}(0))  \right|&\\
&\geq& \theta_\nu(x^*_\nu(0),x^*_{-\nu}(0)) ~&-&	\left|\theta^\varepsilon_\nu(x_\nu^*(\varepsilon),x_{-\nu}^*(\varepsilon))  - 	\theta^\varepsilon_\nu(x_{\nu}^*(0),x_{-\nu}^*(\varepsilon))  \right|&\\
& & ~&-&	\left|\theta^\varepsilon_\nu(x_\nu^*(0),x_{-\nu}^*(\varepsilon))  - 	\theta^\varepsilon_\nu(x_{\nu}^*(0),x_{-\nu}^*(0))  \right|& \\
& & ~&-& 	\left|\theta^\varepsilon_\nu(x_\nu^*(0),x_{-\nu}^*(0))  - 	\theta_\nu(x_{\nu}^*(0),x_{-\nu}^*(0))  \right|& ,
\end{align*}
and similar with (\ref{Eq:ContinuityProperties4}-\ref{Eq:ContinuityProperties5}), we have
\begin{align*}
	\theta_\nu^\varepsilon(\hat{x}_\nu,x^*_{-\nu}(\varepsilon)) &\leq& 	\theta_\nu(\hat{x}_\nu,x^*_{-\nu}(0)) ~&+& \left| 	\theta_\nu^\varepsilon(\hat{x}_\nu,x^*_{-\nu}(\varepsilon))- \theta_\nu(\hat{x}_\nu,x^*_{-\nu}(0)) \right|&,\\
	&\leq& 	\theta_\nu(\hat{x}_\nu,x^*_{-\nu}(0)) ~&+& 	\left|\theta^\varepsilon_\nu(\hat{x}_{\nu},x_{-\nu}^*(\varepsilon))  - 	\theta_\nu(\hat{x}_{\nu},x_{-\nu}^*(\varepsilon))  \right|&\\
	& & ~&+&   	\left|\theta_\nu(\hat{x}_{\nu},x_{-\nu}^*(\varepsilon))  - 	\theta_\nu(\hat{x}_{\nu},x_{-\nu}^*(0))  \right|& ,
\end{align*}
Subtracting these inequality expressions yields with \eqref{Eq:ProofDist} and  (\ref{Eq:ContinuityProperties1}-\ref{Eq:ContinuityProperties5})
\begin{align*}
		\theta_\nu^\varepsilon(x^*_\nu(\varepsilon),x^*_{-\nu}(\varepsilon)) - \theta_\nu^\varepsilon(\hat{x}_\nu,x^*_{-\nu}(\varepsilon)) \geq &~ \theta_\nu(x^*_\nu(0),x^*_{-\nu}(0)) -\theta_\nu(\hat{x}_\nu,x^*_{-\nu}(0)) &\\
	&	-	\left|\theta^\varepsilon_\nu(x_\nu^*(\varepsilon),x_{-\nu}^*(\varepsilon))  - 	\theta^\varepsilon_\nu(x_{\nu}^*(0),x_{-\nu}^*(\varepsilon))  \right|&\\
	&	-	\left|\theta^\varepsilon_\nu(x_\nu^*(0),x_{-\nu}^*(\varepsilon))  - 	\theta^\varepsilon_\nu(x_{\nu}^*(0),x_{-\nu}^*(0))  \right|& \\
	&	-	\left|\theta^\varepsilon_\nu(x_\nu^*(0),x_{-\nu}^*(0))  - 	\theta_\nu(x_{\nu}^*(0),x_{-\nu}^*(0))  \right|& \\
	& - 	\left|\theta^\varepsilon_\nu(\hat{x}_{\nu},x_{-\nu}^*(\varepsilon))  - 	\theta_\nu(\hat{x}_{\nu},x_{-\nu}^*(\varepsilon))  \right|&\\
		& -  	\left|\theta_\nu(\hat{x}_{\nu},x_{-\nu}^*(\varepsilon))  - 	\theta_\nu(\hat{x}_{\nu},x_{-\nu}^*(0))  \right|,&\\
		\geq & ~\varrho -\frac{5}{6}\varrho=\frac{\varrho}{6}>0, &
\end{align*}
which contradicts the assumption that to $x^*(\varepsilon)$ is the Nash equilibrium of \ref{Eq:SmoothNEP}.
Therefore the assumption that $x^*(0)$ is not a Nash equilibrium to \ref{NEP22}, which completes the proof.}
\end{proof}
\textcolor{black}{We remark, that this can also be  understood as a constructive existence proof of Nash equilibria of \eqref{NEP22} if $X_\nu$ is compact for all $\nu=1,\dots,N$, as alternative to Theorem \ref{Th:MainExistence}.}

\textcolor{black}{Note, that we needed to assume convergence of the Nash equilibria.
This strong requirement is weakened on the following two corollaries:
\begin{corollary}[Compact Strategy Sets]
	If in addition $X_\nu$ is compact for all $\nu=1,\dots, N$, the convergence of $x^*_k$ and $x^*(0)$ is not an assumption, because \textcolor{black}{there exists} at least one accumulation point.
\end{corollary}
\begin{corollary}[Accumulation Points are Nash Equilibria]
	If the limit of $x^*_k$ is non unique, every accumulation point of the sequence is a Nash equilibrium of the MLFG.
\end{corollary}}

\section{Numerical Algorithms}\label{Sec:5}
In the previous sections, we reformulated the MLFG in (\ref{follower},\ref{leader}) as  smooth Nash game \eqref{Eq:SmoothNEP} with a smoothing parameter $\varepsilon>0$ and developed theory confirming the validity of this approach.
In this section, a computational method is provided which is consistent to the developed theory.

In particular, we  apply a gradient type method and recall corresponding convergence theory.
As alternative we propose a Newton like method.

\subsection{The Method}
Due to Lemma \ref{lem:KKTNEPeps}, every \textcolor{black}{leader's} unique optimal solution is characterized by its KKT system \ref{Eq:KKTepsilon} for a fixed smoothing parameter $\varepsilon$.
\textcolor{black}{The multistrategy vector of these solutions characterizes the unique Nash equilibrium of \eqref{Eq:SmoothNEP}; therefore, we aim to find a primal dual pair $z=(x,\lambda)$ which satisfies the concatenated KKT conditions.
We abbreviate the concatenation with $F^\varepsilon(z)=(F_1^\varepsilon(z),F_2^\varepsilon(z))^\top$ such that 
}\begin{subequations}\label{Eq:KKTjoint}
	\begin{align}
	\begin{split}
	\textcolor{black}{F^\varepsilon_1(z)}=&Qx+c+\frac{1}{2} (L^\top +Q_y^{-1}B^\top )^\top  a  \\  & +\frac{1}{2} \sum\limits_{i=1}^m a_i (L^\top -Q_y^{-1}B^\top )^\top _{:,i} ~ \tilde{\phi_\varepsilon}'\left([(L^\top -Q_y^{-1}B^\top ) x]_i\right)+ \begin{bmatrix}
	\nabla_{x_{1}} g_{1} (x_{1}) \lambda_{1}\\
	\vdots\\
	\nabla_{x_{N}}  g_{N} (x_{N}) \lambda_{N}
	\end{bmatrix}
	\end{split}\label{Eq:KKTjoint1}\\
	\textcolor{black}{F^\varepsilon_2(z)}=&\min\left\{ \begin{bmatrix}\lambda_{1}\\\vdots\\\lambda_{N}
	\end{bmatrix},\begin{bmatrix}-g_{1}(x_{1})\\\vdots\\-g_{N}(x_{N})
	\end{bmatrix}  \right\} \label{Eq:KKTjoint2}
	\end{align}
\end{subequations}
 using the notation $Q=\mathrm{diag}(Q_1,\dots,Q_N)$ and $c=(c_1^\top,\dots,c^\top_N)^\top$.
The roots of this system characterize the Nash equilibrium for a fixed relaxation parameter $\varepsilon$.
With this notation, the KKT system can be equivalently expressed as the minimization of the auxiliary function $\Psi_\varepsilon: \R^{n+\bar{m}}\rightarrow\R_+$ where $\bar{m}=m_1+\dots+m_N$ and
$$\Psi_\varepsilon(z)=\frac{1}{2}\|F^\varepsilon(z)\|^2_2=\frac{1}{2}\left(\|F_1^\varepsilon(z)\|^2_2+\|F_2^\varepsilon(z)\|^2_2\right).$$
The global minimum is obtained for an $z^*$ satisfying $\Psi_\varepsilon(z^*)=0$.
For convergence theory, the Lipschitz property of $\Psi_{\varepsilon}$ is crucial; therefore, we prove it in the following lemma.
\begin{lemma}\label{Lemma:Lipschitz}
	The function $\Psi_\varepsilon$ is locally Lipschitz and directionally differentiable.
\end{lemma}
\begin{proof}
We verify the  properties for each part of the sum separately.

\noindent
(i) $\frac{1}{2}\|F_1(z)\|_2^2\in C^1$, as a composition of $C^1$ functions because $\phi_\varepsilon$ is assumed to be twice differentiable.
Therefore, this part is locally Lipschitz and directionally differentiable. \\
(ii) $\frac{1}{2}\|F_2(z)\|_2^2=\frac{1}{2}\sum\limits^m_{i=1} \min^2\{\lambda_i,-g_i(x)\}=\frac{1}{8}\sum\limits_{i=1}^m \left(\lambda_i-g_i(x)-|\lambda_i+g_i(x)|\right)^2$ is locally Lipschitz as a composition of locally Lipschitz functions.
It is also directionally differentiable as it is also a composition of directionally differentiable functions.
\end{proof}

We are interested in the solution of the system for $\varepsilon$ close to zero.
However, the problem characteristics are poor for very small $\varepsilon$ and we expect bad numerical performance with arbitrary initial values.
Therefore, we propose to solve a sequence of minimization problems:
\begin{equation*}
\min\limits_z ~\Psi_\varepsilon(z)\quad \mathrm{s.t.} ~z\in\R^{n+\bar{m}},
\end{equation*}
 for a decreasing sequence of \textcolor{black}{positive numbers} 	 $(\varepsilon_i)_{i\in\mathbb{N}}$.
 This approach returns a sequence of KKT points $\left(z^*(\varepsilon_{i})\right)_{i\in\mathbb{N}}=\left(x^*(\varepsilon_{i}), \lambda^*(\varepsilon_{i})\right)_{i\in\mathbb{N}}$ whose primal part $\left(x^*(\varepsilon_i)\right)_{i\in\mathbb{N}}$  is  the Nash equilibrium of NEP$(\varepsilon_{i})$.
 We  use the solution $z^*(\varepsilon_{i})$ as initial value for the subsequent solving for $\varepsilon_{i+1}$.

To further increase the quality of the initial values, we propose an update for the primal variables $x$  based on formal Taylor expansion of the map $\varepsilon\mapsto x^*(\varepsilon)$.
We compute the derivative of the objectives of the Nash game with respect to $\varepsilon$ which implicitly characterize $\frac{\partial x}{\partial \varepsilon}$.
For $\nu=1,\dots, N$, we have
\begin{equation*}
	\frac{\mathrm{d}}{\mathrm{d}\varepsilon} \left(  \nabla_{x_\nu} \theta_\nu^\varepsilon (x_\nu(\varepsilon),x_{-\nu}(\varepsilon)) \right) =0,
\end{equation*}
which leads to the following system
\begin{equation*}
	E\frac{\partial x}{\partial \varepsilon}(\varepsilon)=h.
\end{equation*}
Here, we denote $\tilde{\phi}_\varepsilon(\eta)=\Phi(\eta,\varepsilon)$ to emphasize the explicit dependence on $\varepsilon$, then the linear system has the following coefficient matrix
\begin{equation*}
 E=Q
+\frac{1}{2}\sum\limits_{i=1}^m a_i (L^\top-Q_y^{-1}B^\top)^\top_{:,i}(L^\top-Q_y^{-1}B^\top)_{:,i} \frac{\partial^2 {\Phi}}{\partial t^2}((L^\top-Q_y^{-1}B^\top)^\top_{:,i} x,\varepsilon),
\end{equation*}
and the right-hand-side
\begin{equation*}
	h= \frac{1}{2}\sum\limits_{i=1}^m a_i (L^\top -Q_y^{-1}B^\top )^\top _{:,i} \frac{\partial {\Phi}}{\partial\varepsilon}((L^\top-Q_y^{-1}B^\top)^\top_{:,i} x,\varepsilon).
\end{equation*}
We remark that $E$ is nonsingular since it is composed of the second derivatives of the strictly convex objectives.
We summarize the general approach in the following algorithm.

\begin{algorithm}[H]
	\caption{}\label{Alg:Outer}
	\begin{algorithmic}[1]
		\State \textbf{Initialize} Choose $z^0(\varepsilon_0)=(x^0(\varepsilon_0),\lambda^0(\varepsilon_0))\in\R^{n+\bar{m}}$,  $tol>0$, $\varepsilon_0\in(1,2)$, $\gamma\in(0,1)$.
		\For{$i=0,1,\dots$}
		\State Compute Nash equilibrium of (NEP$(\varepsilon_{i})$)  with initial guess $z^0(\varepsilon_{i})=(x^0(\varepsilon_{i}),\lambda^0(\varepsilon_{i}))$ by $$z^*(\varepsilon_{i})=(x^*(\varepsilon_{i}),\lambda^*(\varepsilon_{i}))=\arg\min\limits_{z\in\R^{n+\bar{m}}} \Psi_{\varepsilon_i}(z),$$ \label{Alg:General3}
		\State Decrease $\varepsilon_{i+1}=\gamma\varepsilon_{i}$,
		\State Compute Taylor update $$d^i=\frac{\partial x}{\partial \varepsilon}(\varepsilon_{i+1} ),$$\label{Alg:General5}
		\State Update initial guess\label{Alg:General6} $x^0(\varepsilon_{i+1})=x^{*}(\varepsilon_{i})-(\varepsilon_i-\varepsilon_{i+1})d^i$ and $\lambda^0(\varepsilon_{i+1})=\lambda^{*}(\varepsilon_{i})$.
		
		\EndFor \State \textbf{end for}
	\end{algorithmic}
\end{algorithm}

In Step \ref{Alg:General5} of the algorithm, we use a forward evaluation of  $\frac{\partial x}{\partial\varepsilon}$ but also $\frac{\partial x}{\partial\varepsilon}(\varepsilon_{i})$ is a valid choice.
In the remainder of this section, we propose two algorithms for computation of the Nash equilibria in Step \ref{Alg:General3}, but other approaches are conceivable, \textcolor{black}{ e.g. diagonalization methods as in  \cite{HR07} or path following techniques \cite{Dirkse1995,Ferris1999}.}

\subsection{Subgradient Method}
To generate the sequence of Nash equilibria, we propose a method which is based on subgradient \textcolor{black}{descent}.
We apply the method of \cite{Bagirov2013} for a fixed smoothing parameter $\varepsilon>0$.
Stationary points of $\Psi_\varepsilon$ are computed as the limit of a sequence of $h$-$\delta$-stationary points.
Bagirov et. al. \cite{Bagirov2013} showed that the limit is a Clarke stationary point.
With this method we obtain the unique Nash Equilibrium of the smoothed game.

Before stating the algorithm and the inherent convergence results, we introduce some terms.

\begin{definition}[$h$-$\delta$ Stationary Point]\label{Def:Wh}
	Let $W_h(x)$ denote the closed convex hull of all possible quasisecants of a locally Lipschitz function $f:\R^n\rightarrow \R$ at the point $x\in\R^n$ with length $h>0$:
	$$W_h(x)=\overline{\mathrm{conv}}\left\{w\in\R^n: \exists d\in\R^n \text{ with } \|d\|=1: w=v(x,d,h)   \right\}.$$
	Then a point $x$ is called a $h$-$\delta$ stationary point of a locally Lipschitz function $f:\R^n\rightarrow \R$ if and only if
	$$\min\left\{\|v\|: v\in W_h(x)  \right\}<\delta.$$
\end{definition}

\begin{lemma}[Termination]
	(1) If $\max\{\|v\|: v\in W_h(z)\}<\infty$ for all iterates $z^k\in\R^{n+\bar{m}}$, the loop in Lines 7-15 terminates after finitely many iterations with a decent direction.
	(2) The loop in Lines 4-20 terminates after finitely many iterations with a $h$-$\delta$-stationary point. 
\end{lemma}
\begin{proof}
	(1) Since $\Psi_\varepsilon$ is locally Lipschitz with Lemma \ref{Lemma:Lipschitz}, \cite[Proposition 4.1]{Bagirov2013} is applicable.\\
	(2) The function $\Psi_\varepsilon$ is bounded from below as it takes nonnegative values only, therefore \cite[Proposition 5.1]{Bagirov2013} is applicable.
\end{proof}

\begin{theorem}[Convergence]
	Assume $\mathcal{L}(z^0)=\left\{z\in\R^{n+\bar{m}}: \Psi_{\varepsilon}(z)\leq\Psi_\varepsilon(z^0)  \right\}$ is bounded and Assumption \ref{Assumption31} is fulfilled.
	Then there exists at least one accumulation point of the sequence $(z^k)_{k\in\mathbb{N}}$ generated by Alg. 2 and any accumulation point is a stationary point of $\Psi_\varepsilon$.
\end{theorem}
\begin{proof}
	Due to Lemma \ref{Lemma:Lipschitz}, $\Psi_{\varepsilon}$ is locally Lipschitz and therefore, \cite[Proposition 5.2]{Bagirov2013} is applicable.
	The boundedness of $\mathcal{L}(z^0)$ implies that \textcolor{black}{there exists} at least one accumulation point.
\end{proof}
Bagirov et. al. \cite{Bagirov2013} state that subgradients are in particular quasisecants and therefore,  we limit ourselves to the usage of subgradients as decent directions and to $h=0$ in the implementations.
The algorithm is stated as Algorithm \ref{Alg:Subgr} below.
\begin{algorithm}[H]
	\caption{Subgradient Method}\label{Alg:Subgr}
	\begin{algorithmic}[1]
		\State \textbf{Initialize} $h_0>0$ , $\delta_0>0$, $\gamma\in(0,1)$, $z^0\in\R^{n+\bar{m}}$, $d_0\in\R^{n+\bar{m}}$ with $\|d_0\|=1$, $0<c_2\leq c_1\leq 1$, $\varepsilon>0$.

		\For{$k=0,\dots$}
		\State $\bar{z}_1=z^k$,
		\For{$j=1,\dots$}\Comment{compute $h$-$\delta$-stationary point}
		\State Compute quasisecant $v_0=v(\bar{z}_j,d_0,h)$,
		\State $\tilde{v}_0=v_0$,
		\For{$i=0,1,\dots$}\Comment{find decent direction}
		\State $c_i=\arg\min\{\|cv_i+(1-c)\tilde{v}_i\|^2_2~ | c\in(0,1)\}$,
		\State $\bar{v}_i=c_iv_i+(1-c_i)\tilde{v}_i$,
		\If{$\|\bar{v}_i\|\leq \delta_k$}  return $v^j=\bar{v}_i$  .\EndIf
		\State $d_i=-\frac{\bar{v}_i}{\|\bar{v}_i\|}$,
		\If{$\Psi_\varepsilon(\bar{z}_j+hd_{i})-\Psi_\varepsilon(\bar{z}_j)\leq -c_1h\|\bar{v}_i\|$} return $v^j=\bar{v}_i$. \EndIf
		\State Compute quasisecant $v_{i+1}=v(x,d_i,h)$,
		\State $\tilde{v}_{i+1}=\bar{v}_i$.
		\EndFor \State \textbf{end for}
		\If{$\|v^j\|\leq \delta_k$}  Stop.  \EndIf
		\State $d^j=-\frac{v^j}{\|v^j\|}$,
		\State Compute step length such that
		$$\sigma_j=\arg\max\{\Psi_\varepsilon(\bar{z}_j+\sigma d^j)-\Psi_\varepsilon(\bar{z}_j)\leq -c_2\sigma\|v^j\|~|\sigma>0\},$$
		\State Update $\bar{z}_{j+1}=\bar{z}_j+\sigma_j d^j$.
    
		\EndFor \State \textbf{end for}
		\State $z^{k+1}=\bar{z}_j$,
		\State $h_{k+1}=\gamma h_k$, 
		\State $\delta_{k+1}=\gamma \delta_k $.
		\EndFor \State \textbf{end for}

	\end{algorithmic}
\end{algorithm}

\subsection{Nonsmooth Newton Method}
Next, we present an improved method.
The joint KKT system  \eqref{Eq:KKTjoint} leads to the problem to find the unique $z^*(\varepsilon)=(x^*(\varepsilon),\lambda^*(\varepsilon))$ that satisfy
\begin{equation*}
	F^\varepsilon(z)=0.
\end{equation*}
This  is a nonlinear and nonsmooth system of equations which depend on the  parameter $\varepsilon>0$.
The generalized Newton method can be written as the solving of a sequence of the linear systems
\begin{equation*}
H \left(z^{k+1}-z^k\right) = -F^\varepsilon(z^k),
\end{equation*}
for an element $H\in\partial F^\varepsilon(z^k)$ of the Clarke subdifferential of $F^\varepsilon$.
The explicit structure of a generalized Jacobian $H$ can be found in \ref{App:Newton}.

Since $H$ is not necessarily regular,  we verify this property in Step \ref{Alg:Newton5} of  Algorithm \ref{Alg:Newton} and use a first order decent direction if necessary.
This subgradient decent also serves as globalization strategy.
\begin{algorithm}[H]
	\caption{Nonsmooth Newton Method}\label{Alg:Newton}
	\begin{algorithmic}[1]
		\State \textbf{Initialize} Choose $z^0=(x^0,\lambda^0)\in\R^{n+\bar{m}}$,  $\beta\in(0,1)$, $\sigma\in(0,0.5)$, $tol>0$, $\varepsilon>0$.
		
		\For{$k=0,\dots$}
		\If{$\Psi_{\varepsilon}\leq tol$}  Stop. \EndIf
		\State Let $H\in\partial F^\varepsilon(z^k)$,
		\If{$H$ singular} do subgradient decent of $\Psi_{\varepsilon}$, thus choose $$s^k\in-\partial\Psi_{\varepsilon}(z^k),$$\label{Alg:Newton5}
		and the step length $t_k=\max\{\beta^l|l=0,1,\dots \}$ which fulfills the Armijo condition 
		$$\Psi_{\varepsilon}(z^k+t^ks^k)\leq \Psi_{\varepsilon}(z^k) + t^k \sigma  {s^k}^\top s^k,$$
		\Else ~ let $t_k=1$ and compute Newton step by solving  $$ Hs^k=-F^{\varepsilon}(z^k),$$
		\EndIf
		\State Update $z^{k+1}=z^k+t_ks^k$.
		\EndFor \State \textbf{end for}
	\end{algorithmic}
\end{algorithm}
For further discussions and convergence analysis we refer to e.g. \cite{Qi1993}.

\section{Numerical Results}\label{Sec:6}
In the previous sections, we proposed an algorithm with gradient updates of the primal variables.
This included the computation of Nash equilibria for a sequence of smoothing parameter $(\varepsilon_i)_{i\in \mathbb{N}}$.
For this computation, we introduced a subgradient and a Newton method.
The presented numerical results are obtained for the data sets in \ref{App:Data} which are adapted from \cite{HuFu2013}.
All plots are generated for the Data Set 1, however experiments with Data Set 2 produced similar graphics.

The naive approach of computing a sequence of Nash equilibria is to use the Nash equilibrium of a larger smoothing parameter as initial for the subsequent computation with the smaller smoothing parameter.
The main purpose of the outer Taylor expansion based update in Algorithm \ref{Alg:Outer} (Step \ref{Alg:General5} and \ref{Alg:General6}) is to improve the quality of the initials in order to reduce the computational effort in Step \ref{Alg:General3}.

In the upper left part of Figure \ref{Fig:Fig1}, we observe the quadratic decent of the error for decreasing smoothing parameter.
In the upper right part, the Taylor update is exemplary illustrated for one component of the leader variables.
The blue dots indicate each the Nash equilibrium of a NEP$(\varepsilon_i)$, $x^*(\varepsilon_i)$.
A black line represents the Taylor update and the lower end of a black line indicates the updated initial values $x^0(\varepsilon_{i+1})$ for the subsequent Nash equilibrium computation.

The lower part of Figure \ref{Fig:Fig1} is dedicated to illustrate the importance of large smoothing parameter for the first computations of Nash equilibria.
Since the problem gets closer to it original nonsmooth formulation as $\varepsilon$ decreases, the problem is also more challenging to solve for both Subgradient and Nonsmooth Newton  method.
We observe this expected behavior, in particular if we compare the number of iterations for $\varepsilon=1.6$ and $\varepsilon=0.1$.

\begin{figure}
	\begin{minipage}{0.5\textwidth}
		\begin{center}
			\includegraphics[width=\textwidth]{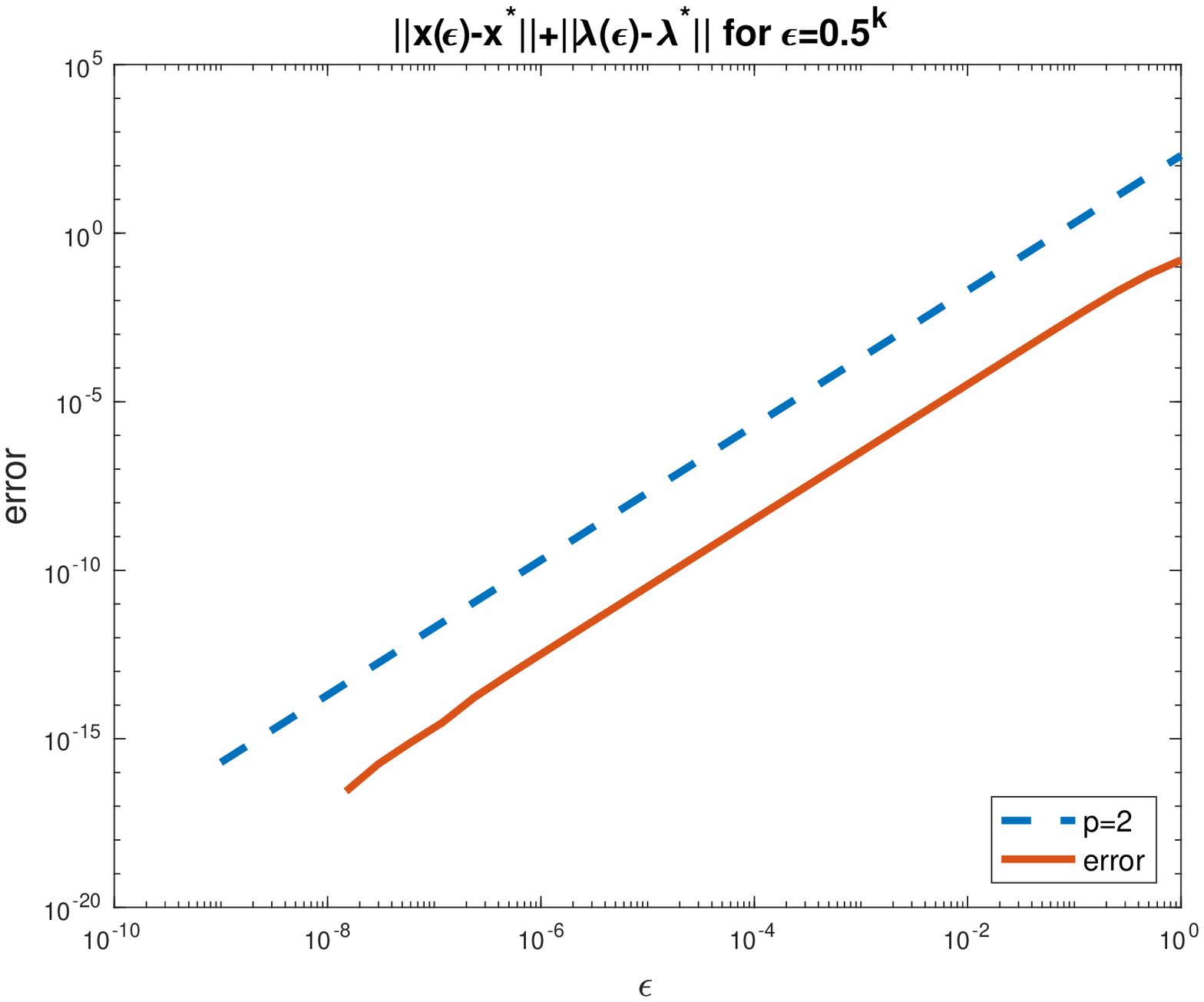} \\
			\includegraphics[width=\textwidth]{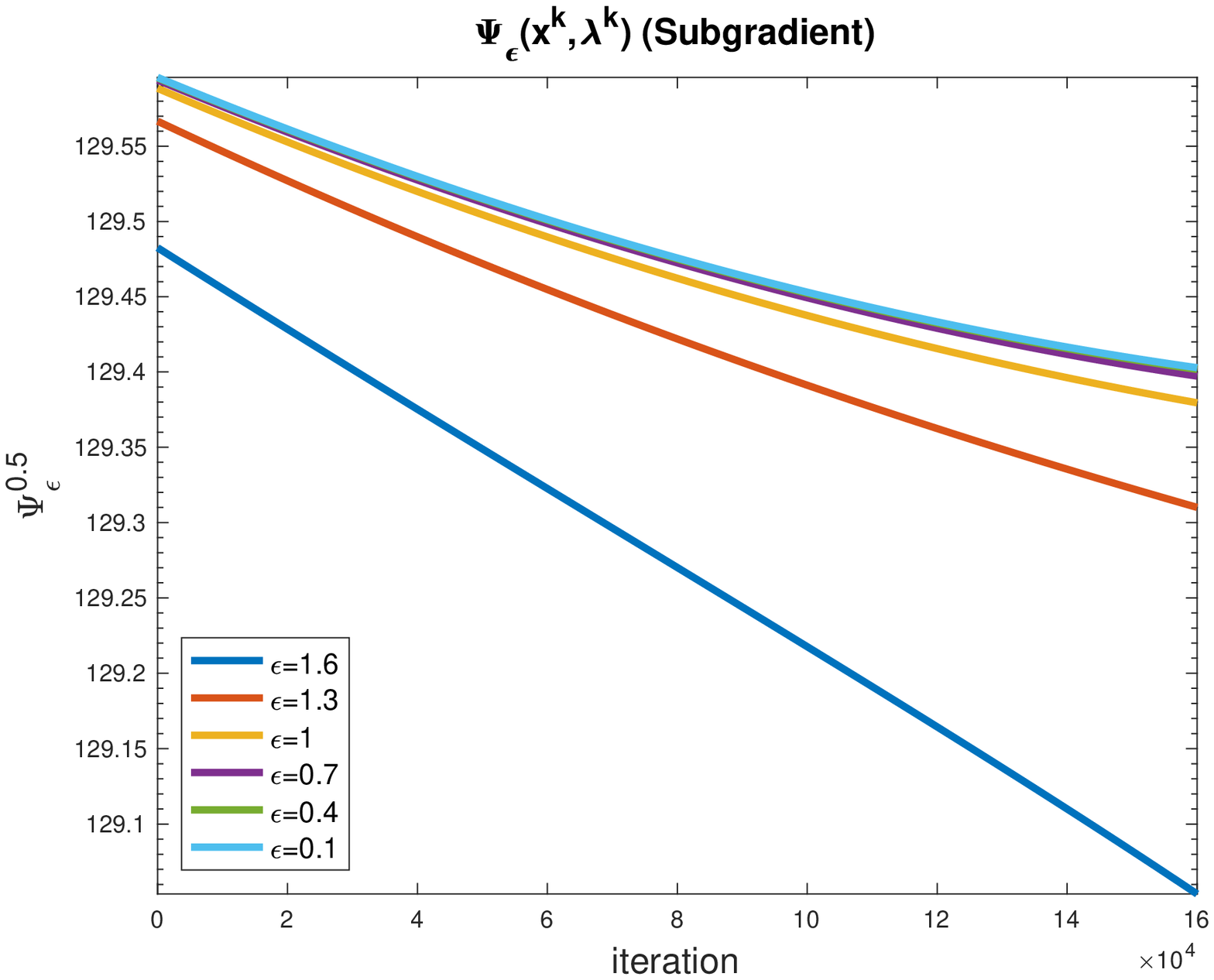} 
		\end{center}
	\end{minipage}\begin{minipage}{0.5\textwidth}
			\begin{center}
		\includegraphics[width=\textwidth]{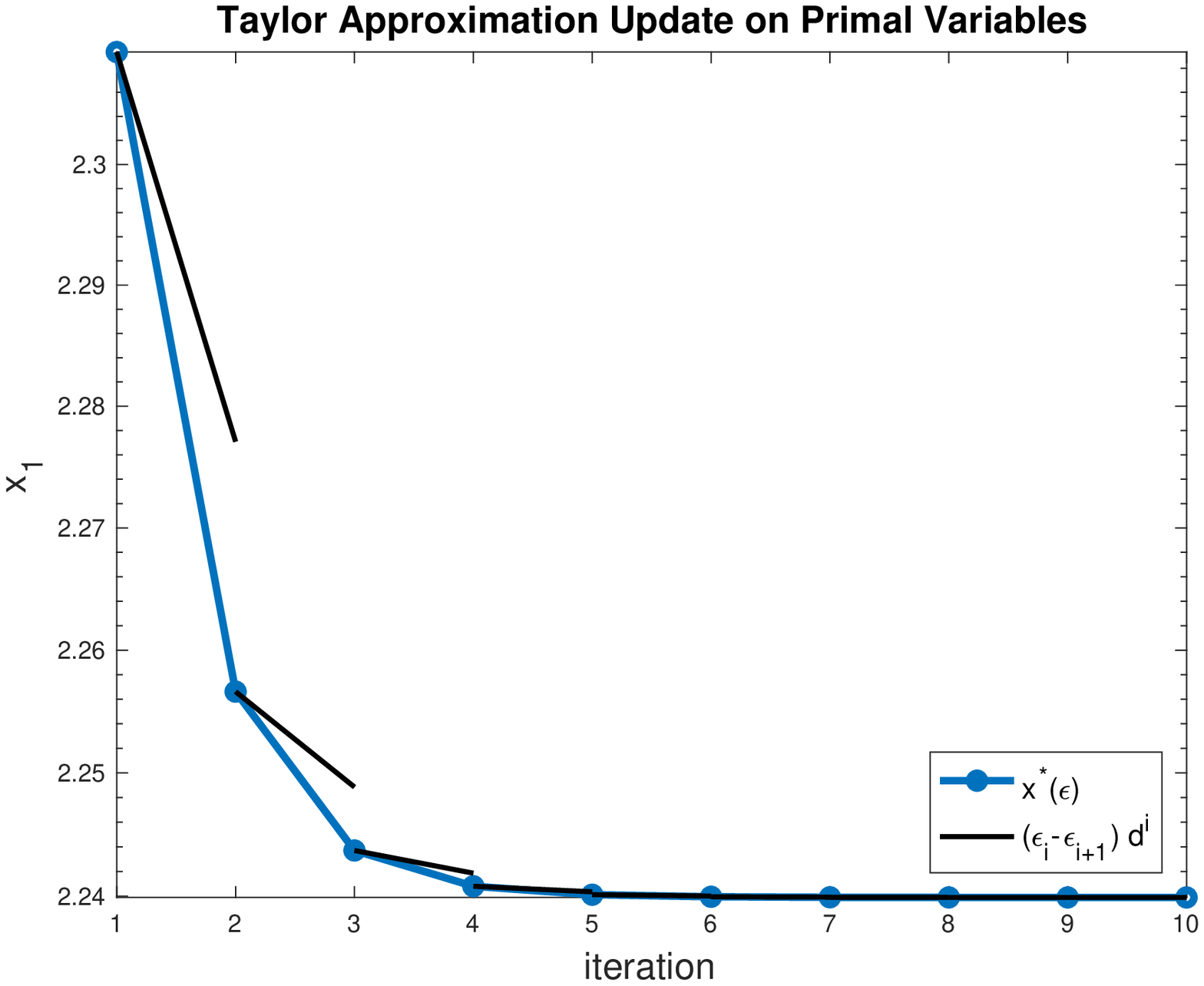} \\
		\includegraphics[width=\textwidth]{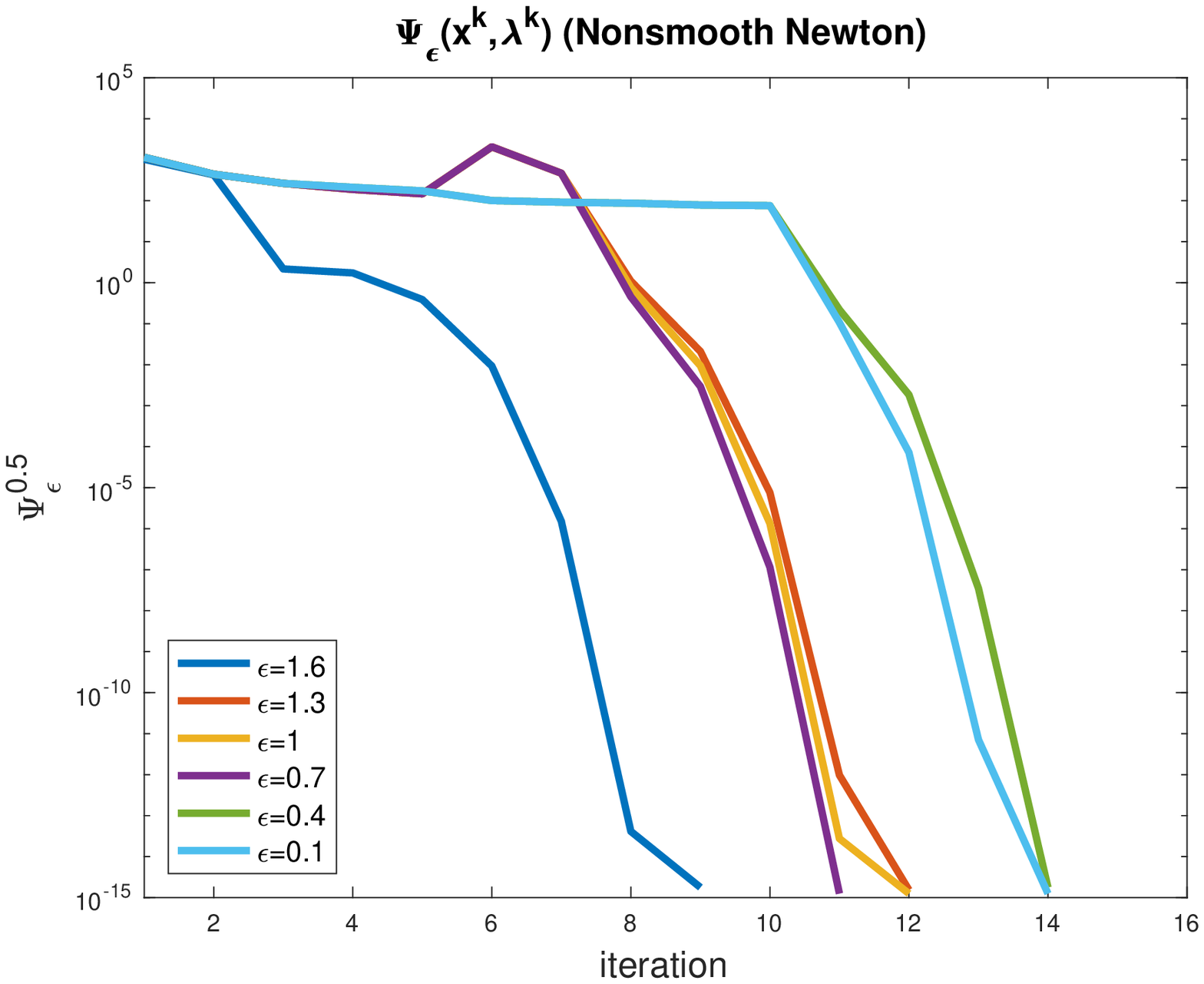} 
	\end{center}
\end{minipage}
	\caption{\label{Fig:Fig1}Upper left: Quadratic convergence to the limit Nash equilibrium $x^*(0)$, right: Taylor expansion based update on primal variables, exemplary for first leader variable; Lower: comparison of Subgradient and Nonsmooth Newton method for varying smoothing parameter. }
\end{figure}

As already seen in Figure \ref{Fig:Fig1}, the subgradient based method suffers from characteristically slow convergence for our instances.
In Figure \ref{Fig:Fig2}, all iterations for a sequence of decreasing smoothing parameter are shown.

\begin{figure}[H]
\begin{center}
	\includegraphics[width=0.5\textwidth]{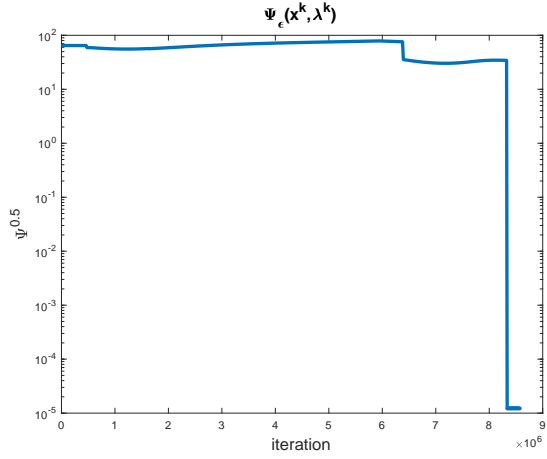} 
\end{center}

	\caption{\label{Fig:Fig2}Subgradient Method; all iterations for decreasing sequence of smoothing parameters.}
\end{figure}

Similarly to Figure \ref{Fig:Fig2},  left in Figure \ref{Fig:Fig3}, all iterations for a sequence of decreasing smoothing parameter are shown for the Nonsmooth Newton.
The alternating behavior is due to the decreasing parameter changing the minimization problem.
The right part of Figure  \ref{Fig:Fig3} illustrates the decrease in $\Psi_\varepsilon$ for different random initial values but fixed smoothing parameter.

\begin{figure}[H]
	\label{Fig:Fig3}
	\begin{minipage}{0.5\textwidth}
		\begin{center}
			\includegraphics[width=\textwidth]{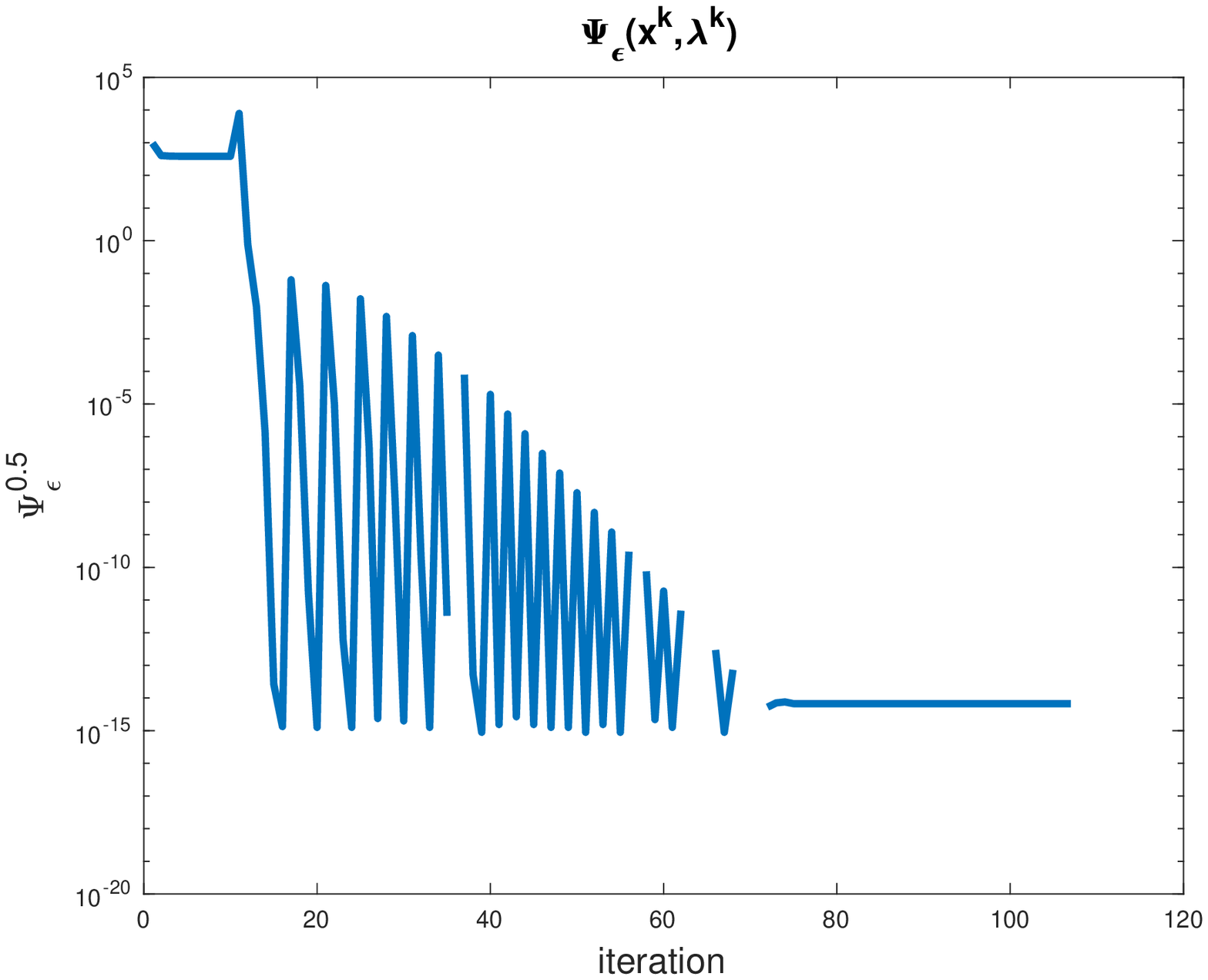} 
		\end{center}
	\end{minipage}\begin{minipage}{0.5\textwidth}
		\begin{center}
			\includegraphics[width=\textwidth]{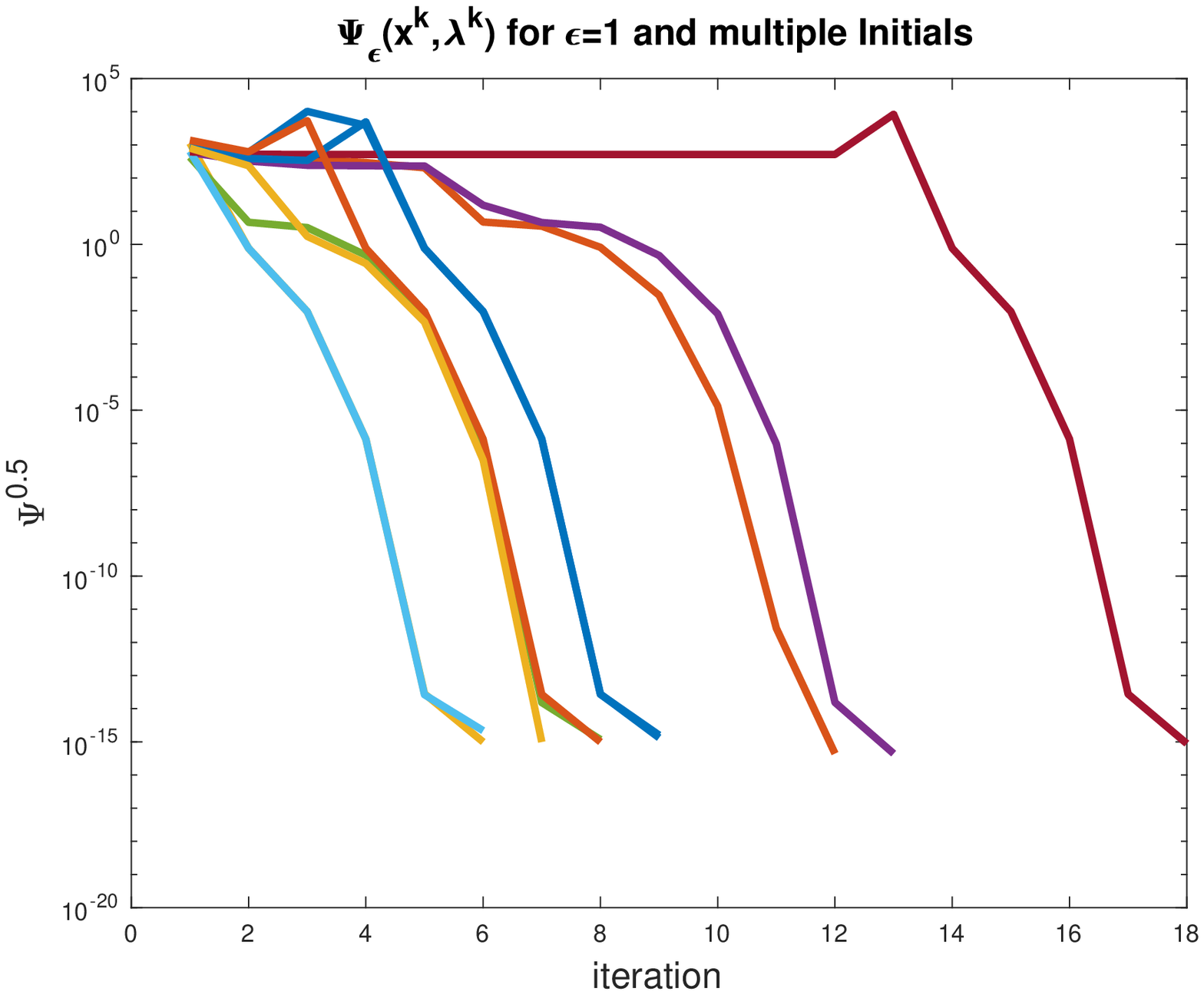} 
		\end{center}
	\end{minipage}
	
	\caption{Nonsmooth Newton Method; left all iterations for decreasing sequence of smoothing parameters, right for one smoothing parameter and multiple initials.}
\end{figure}

\section{Conclusion and Outlook}
We presented a quadratic MLFG and explicitly computed the best response of the follower player.
With this best response we derived a Nash game formulation where existence theory is available.
Furthermore, we smoothed the best response function and formulate the MLFG as smooth Nash game and proved existence and uniqueness of the Nash equilibrium for all smoothing parameters.
We followed an all KKT approach to characterize the corresponding Nash equilibrium.
For decreasing positive smoothing parameter, we showed that the limit of Nash equilibria satisfies  the conditions of \textcolor{black}{S-stationarity}.
Further, we demonstrated that \textcolor{black}{S-stationary} points are eventually Nash equilibria of the MLFG.
Numerically, we computed Nash equilibria with a globalized nonsmooth Newton and compare with a standard methods based on subgradients.
For efficient computation, we updated the primal variables by a Taylor approximation before a subsequent computation of the Nash equilibrium for a smaller smoothing parameter.

\textcolor{black}{The numerical comparison to alternative methods designed for MPECs and EPECs are planned as future research.}
Also, the extension possibilities discussed in Appendix \ref{Sec:Ext} are  subject of further investigation.

\appendix
\section{Appendix}
\subsection{Remarks on Extensions to the Model}\label{Sec:Ext}
In the following, we mention two generalizations to the follower problem in \eqref{follower}, which seem obvious to include.
We explain challenges to motivate future research.

In  \eqref{follower}, it is assumed that the matrix $Q_y$ is positive definite and diagonal.
A popular approach in quadratic programming is to enforce diagonality of symmetric positive definite matrices by introducing an auxiliary variable $z=D^\top y$, where we have the Cholesky decomposition $Q_y=DD^\top$, which exists for symmetric  positive definite  $Q_y$.
The follower's problem can be equivalently formulated as a minimization problem in $z$:
\begin{equation*}
\min\limits_{z\in \R^m} \frac{1}{2}z^\top z-b(x)^\top D^{-\top}z \quad \st\quad  z\geq D^\top l(x).
\end{equation*}
As in Lemma \ref{Th:bestresponse}, we   derive the solution to this optimization problem explicitly:
\begin{equation*}
z(x)=\max\{D^{-1}b(x),D^\top l(x)\},
\end{equation*}
and we recover the follower's best response in this setting:
\begin{equation*}
y(x)=D^{-\top}\max\{D^{-1}b(x),D^\top l(x)\}.
\end{equation*}
However, unlike \eqref{solY2}, this function is not necessarily componentwise convex.
But this property is essential to guarantee convexity of the leaders' objectives.
This property is required  for the proof of existence of Nash equilibria in Theorem \ref{Th:MainExistence} as it is based on Kakutani fixed-point theorem.
An other obvious extension to the follower problem is to incorporate upper bounds besides the discussed lower bounds, i.e. $l(x)\leq y\leq u(x)$.
We can also derive an explicit representation of the best response, e.g. via projection of the objective's gradient into the feasible set.
For that let $\tilde{y}_i(x)=(Q_y)_{ii}^{-1}b_i(x)$, then the best response is for $i=1,\dots,m$
\begin{equation*}
y_i(x)=\mathrm{median}( l_i(x),\tilde{y}_i(x),u_i(x))=\begin{cases}

l_i(x), & \tilde{y}_i(x)\leq l_i(x),\\
u_i(x), &  \tilde{y}_i(x)\geq u_i(x),\\
\tilde{y}_i(x), & \text{else}.
\end{cases}
\end{equation*}
Like in the previous case, this best response is not necessarily convex; therefore, the existence result in Theorem \ref{Th:MainExistence} does not apply.

\subsection{Subgradient Method}
In the following, we state the definition of quasisecants and a theorem which relates quasisecants and subgradients.
Furthermore, we state a assumption which is needed in the convergence theorem of the subgradient method.
All is adapted from \cite{Bagirov2013} and can be found there in an extended form.

\begin{definition}[Quasisecant]
	A vector $v=v(x,d,h)\in\R^n$ is called a quasisecant of a locally Lipschitz function $f:\R^n\rightarrow \R$ at the point $x\in\R^n$ in direction $d\in\R^n$ with $\,d\,=1$ with the length $h>0$ if and only if
	\begin{equation*}
		f(x+hd)-f(x)\leq h\langle v,d\rangle,
	\end{equation*}
	and
		\begin{equation*}
	v\in\partial_{d,h}f(x)+B_{O(h)},
	\end{equation*}
	where  $\partial_{d,h} f=\cup_{t\in[0,h]}\partial f(x+td)$ denotes the union of all Clarke subdifferentials over the set $\mathrm{conv}(x,x+hd)$ and $B_{O(h)}$ denotes a ball for which $O(h)\rightarrow0$ for $h\rightarrow0$. 
\end{definition}

For the convergence proof it is necessary to study the relation of $W_h(x)$ (Definition \ref{Def:Wh}) and the subdifferential $\partial f(x)$ and therefore, the following assumption is crucial.
\begin{assumption}\label{Assumption31}
	At any given point $x\in\R^n$ there exists $\delta=\delta(x)>0$ such that $O(y,h)\downarrow0$ uniformly as $h\downarrow0$ for all $y\in B_\delta(x)$ that is for any $\eta>0$ there exists $h(\eta)>0$ such that $O(y,h)<\eta$ for all $h\in (0,h(\eta))$ and $y\in B_\delta(x)$.
\end{assumption}
In particular, this assumptions guarantees a certain relation between quasisecants and subgradients.
\begin{theorem}
	Assume that a function $f$ satisfies Assumption \ref{Assumption31}.
	Then at a given point $x\in\R^n$ for any $\eta>0$ there exists $\delta=\delta(\eta)$ and $h(\eta)>0$ such that
	\begin{equation*}
		W_h(y)\subset \partial f+ B_\eta,
	\end{equation*}
	for all $h\in(0,h(\eta))$ and $y\in B_\delta(x)$.
	Furthermore, it holds for locally Lipschitz function that  the limit as $h\rightarrow 0$ of the $W_h(x)$ lies in the subdifferential, i.e.
	\begin{equation*}
		W_0(x)\subset \partial f(x).
	\end{equation*}
\end{theorem}

%
%

\subsection{Nonsmooth Newton}\label{App:Newton}
We propose a nonsmooth Newton method to compute Nash equilibria.
In order to keep the readability of the paper, we specify the structure of the generalized Jacobian here.

We look at the elements of $\partial F_\varepsilon$ as a block matrix:
$$H=\begin{bmatrix}
\boxed{A } & \boxed{B}\\\boxed{C}&\boxed{D}
\end{bmatrix},$$
where the block $\boxed{A }\in\R^{n\times n}$ is
\begin{align*}
\nonumber\boxed{A}&=Q+ \begin{bmatrix}
\nabla_{x_{(1)}}(\nabla_{x_{1}} g_{1}(x_{1}) \lambda_{1}) & & \\
&\ddots & 
\\
& & 	\nabla_{x_{N}}(\nabla_{x_{N}} g_{N}(x_{N}) \lambda_{N})
\end{bmatrix}\\
& +\frac{1}{2}\sum\limits_{i=1}^m a_i (L^\top -Q_y^{-1}B^\top )^\top _{:,i}(L^\top -Q_y^{-1}B^\top )_{:,i} \\ & \left(\frac{1}{\sqrt{[(L^\top -Q_y^{-1}B^\top ) x]_i^2 +4\varepsilon^2} } -\frac{[(L^\top -Q_y^{-1}B^\top )x]_i^2}{\sqrt{[(L^\top -Q_y^{-1}B^\top ) x]_i^2 +4\varepsilon^2}^3}\right),
\end{align*}
the block $\boxed{B}\in\R^{n\times\bar{m}}$ is
\begin{equation*}
\boxed{B}=\begin{bmatrix}
\nabla_{x_{1}} g_{1}(x_{1}) & & \\
&\ddots & \\
& & \nabla_{x_{N}} g_{N}(x_{N}) 
\end{bmatrix},
\end{equation*}
and
the block diagonal $\boxed{C}\in \R^{\bar{m}\times n}$ is
\begin{equation*}
\boxed{C}=\begin{bmatrix}
\boxed{C_1} & & \\ & \ddots & \\ & & \boxed{C_N}
\end{bmatrix},
\end{equation*}
with the blocks $\boxed{C_\nu}\in\R^{m_\nu\times n_\nu}$ and the entries
\begin{equation*}
\left(\boxed{C_\nu}\right)_{i,j}=\partial^C_{x_\nu^j} \min\left\{ \lambda_\nu^i, -g_\nu^i(x_\nu)  \right\}=\begin{cases} 
0, & \lambda_\nu^i <-g_\nu^i(x_\nu), \\
-\frac{\partial}{\partial x_\nu^j} g_\nu^i(x_\nu), & \lambda_\nu^i > -g_\nu^i(x_\nu), \\
\left[0, -\frac{\partial}{\partial x_\nu^j} g_\nu^i(x_\nu) \right], & \lambda_\nu^i = -g_\nu^i(x_\nu). \\
\end{cases}
\end{equation*}
and the diagonal matrix: 
\begin{equation*}
\boxed{D}=\begin{bmatrix}
\boxed{D_1} & & \\ & \ddots & \\ & & \boxed{D_N}
\end{bmatrix} \in\R^{\bar{m}\times\bar{m}},
\end{equation*}
with its blocks $\boxed{D_\nu}\in\R^{m_\nu\times m_\nu}$ with the entries:
\begin{equation*}
\left(\boxed{D_\nu}\right)_{i}=\partial^C_{\lambda_\nu^i} \min\left\{ \lambda_\nu^i, -g_\nu^i(x_\nu)  \right\}=\begin{cases} 
1, & \lambda_\nu^i <-g_\nu^i(x_\nu), \\
0, & \lambda_\nu^i > -g_\nu^i(x_\nu), \\
\left[1,0 \right], &  \lambda_\nu^i = -g_\nu^i(x_\nu). \\
\end{cases}
\end{equation*}

\subsection{The Data}
\label{App:Data}
In the following, we specify the data used for the experiments presented in Section 5.
We adapted the data used in \cite{HuFu2013}.

\subsubsection{Data Set 1}
We consider $N    = 2$ leader with each $n_1=n_2=2$ variables.
The objectives of the leader are given by
\begin{equation*}
Q_{1}= \begin{bmatrix}
1.7 & 1.6 \\
1.6 & 2.8 
\end{bmatrix}, \quad Q_2=\begin{bmatrix} 2.7 & 1.3 \\ 1.3 & 3.6 
\end{bmatrix}, \quad c_1=c_2=\begin{bmatrix} 0\\ 0\end{bmatrix},  \quad a=\begin{bmatrix} 1.4 \\ 2.6 \\ 2.1\end{bmatrix}.
\end{equation*}
Each leader has $m_1=m_2=3$ linear constraints $g_\nu=A_\nu^Tx_\nu + b_\nu\leq 0 $ with 
\begin{align*}
A_1=\begin{bmatrix}
1.6  & 0.8 & 1.3 \\
2.6 &  2.2 & 1.7
\end{bmatrix}, \quad	b_1=\begin{bmatrix}
1.6 \\ 1.2 \\ 0.4 
\end{bmatrix},\qquad
A_2=\begin{bmatrix}
1.8 & 1.6 & 1.4 \\
1.3 & 1.2 & 2.7
\end{bmatrix}, \quad	b_2=\begin{bmatrix}
1.6 \\ 1.5 \\ 2.6
\end{bmatrix}.
\end{align*}

The follower has $M=3$ variables and its objective and constraints are given by
\begin{equation*}
Q_y=\begin{bmatrix}
2.5 &0& 0\\
0& 3.6& 0 \\
0 &0 &4.6
\end{bmatrix}, \quad B= \begin{bmatrix}
2.3& 1.4 &2.6 \\
1.3& 2.1 &1.7 \\
2.5& 1.9& 1.4 \\
1.3 &2.4 &1.6
\end{bmatrix}, \quad L=\begin{bmatrix}
1.3& 2.4& 1.8 \\
1.3 &2.4& 1.8 \\
1.3& 2.4& 1.8\\
1.3& 2.4& 1.8 
\end{bmatrix}. 
\end{equation*}

\subsubsection{Data Set 2}
We consider $N    = 3$ leader with each $n_1=n_2=n_3=2$ variables.
The objectives of the leader are given by
\begin{align*}
Q_{1}= \begin{bmatrix}
2.5 &1.6 \\
1.6 &3.8  \end{bmatrix}, ~ Q_2=\begin{bmatrix} 2.9 &1.3 \\
1.3 &1.8   
\end{bmatrix},~ Q_3=\begin{bmatrix}
3.2 &2.3 \\
2.3 &2.6  
\end{bmatrix}, ~ c_1=c_2=c_3=\begin{bmatrix} 0\\ 0\end{bmatrix}, ~  a=\begin{bmatrix} 0.4 \\1.6\\ 2.6\end{bmatrix}.
\end{align*}
Each leader has $m_1=m_2=m_3=3$ linear constraints $g_\nu=A_\nu^Tx_\nu + b_\nu\leq 0 $ with 
\begin{align*}
A_1&=&\begin{bmatrix}
1.6& 0.8 &1.3\\
2.6& 2.2 &1.7
\end{bmatrix}, ~~
A_2&=&\begin{bmatrix}
1.8 & 1.6 & 1.4 \\
1.3 & 1.2 & 2.7
\end{bmatrix}, ~~
A_3&=&\begin{bmatrix}
2.3 &1.9& 1.6 \\
1.3 &1.7& 2.7
\end{bmatrix}, \\
b_1&=&\begin{bmatrix}
1.6&1.2 & 0.4
\end{bmatrix}^\top, ~~
	b_2&=&\begin{bmatrix}
1.6& 1.5 & 2.6
\end{bmatrix}^\top,~~
	b_3&=&\begin{bmatrix}
1.5& 0.3 &1.8
\end{bmatrix}^\top.
\end{align*}

The follower has $M=3$ variables and its objective and constraints are given by
\begin{equation*}
Q_y=\begin{bmatrix}
3.7 &0 &0 \\
0& 2.6& 0 \\
0& 0& 0.7
\end{bmatrix}, \quad B= \begin{bmatrix}
0.8& 2.1 &1.3 \\
1.5 & 2.3 &0.7 \\
1.5 & 0.9 &2.4 \\
1.8 &2.3& 3.6 \\
1.3 &1.7 &1.7 \\
1.1 &2.6& 1.6 \\
\end{bmatrix}, \quad L=\begin{bmatrix}
0.8 & 2.1 & 1.3 \\
1.5 &2.3 & 0.7 \\
1.5 &0.9& 2.4 \\
1.8 &2.3& 3.6 \\
0.5 &1.1& 2.1 \\
1.2 &1.5& 1.8 
\end{bmatrix}. 
\end{equation*}

\section*{Acknowledgements}
Special thanks to Michael Ferris and Olivier Huber for fruitful discussions on model extensions.
This work was supported by the DFG under Grant STE2063/2-1. 

\bibliographystyle{abbrvnat}
\bibliography{MLFGarXiv.bib}

\begin{thebibliography}{30}
\providecommand{\natexlab}[1]{#1}
\providecommand{\url}[1]{\texttt{#1}}
\expandafter\ifx\csname urlstyle\endcsname\relax
  \providecommand{\doi}[1]{doi: #1}\else
  \providecommand{\doi}{doi: \begingroup \urlstyle{rm}\Url}\fi

\bibitem[Allevi et~al.(2018)Allevi, Aussel, and Riccardi]{Allevi2018}
E.~Allevi, D.~Aussel, and R.~Riccardi.
\newblock On an equilibrium problem with complementarity constraints
  formulation of pay-as-clear electricity market with demand elasticity.
\newblock \emph{Journal of Global Optimization}, 70\penalty0 (2):\penalty0
  329--346, Feb 2018.
\newblock ISSN 1573-2916.

\bibitem[Aussel and Svensson(2018)]{Aussel2018}
D.~Aussel and A.~Svensson.
\newblock Some remarks about existence of equilibria, and the validity of the
  epcc reformulation for multi-leader-follower games.
\newblock \emph{Journal of nonlinear and convex analysis}, 19\penalty0
  (7):\penalty0 1141--1162, 2018.

\bibitem[Aussel et~al.(2016)Aussel, Cervinka, and Marechal]{Aussel16}
D.~Aussel, M.~Cervinka, and M.~Marechal.
\newblock Deregulated electricity markets with thermal losses and production
  bounds: models and optimality conditions.
\newblock \emph{RAIRO-Oper. Res.}, 50\penalty0 (1):\penalty0 19--38, 2016.

\bibitem[Aussel et~al.(2017{\natexlab{a}})Aussel, Bendotti, and
  Pi\v{s}t\v{e}k]{Aussel17part1}
D.~Aussel, P.~Bendotti, and M.~Pi\v{s}t\v{e}k.
\newblock Nash equilibrium in a pay-as-bid electricity market: Part 1 -
  existence and characterization.
\newblock \emph{Optimization}, 66\penalty0 (6):\penalty0 1013--1025,
  2017{\natexlab{a}}.

\bibitem[Aussel et~al.(2017{\natexlab{b}})Aussel, Bendotti, and
  Pi\v{s}t\v{e}k]{Aussel17part2}
D.~Aussel, P.~Bendotti, and M.~Pi\v{s}t\v{e}k.
\newblock Nash equilibrium in a pay-as-bid electricity market: Part 2 - best
  response of a producer.
\newblock \emph{Optimization}, 66\penalty0 (6):\penalty0 1027--1053,
  2017{\natexlab{b}}.

\bibitem[Bagirov et~al.(2013)Bagirov, Jin, Karmitsa, AlÂ Nuaimat, and
  Sultanova]{Bagirov2013}
A.~M. Bagirov, L.~Jin, N.~Karmitsa, A.~AlÂ Nuaimat, and N.~Sultanova.
\newblock Subgradient method for nonconvex nonsmooth optimization.
\newblock \emph{Journal of Optimization Theory and Applications}, 157\penalty0
  (2):\penalty0 416--435, May 2013.
\newblock ISSN 1573-2878.

\bibitem[DeMiguel and Xu(2009)]{DeMiguel2009}
V.~DeMiguel and H.~Xu.
\newblock \textcolor{black}{A Stochastic Multiple-Leader Stackelberg Model:
  Analysis, Computation, and Application}.
\newblock \emph{Operations Research}, 57\penalty0 (5):\penalty0 1220--1235, oct
  2009.
\newblock \doi{10.1287/opre.1080.0686}.

\bibitem[Dirkse and Ferris(1995)]{Dirkse1995}
S.~P. Dirkse and M.~C. Ferris.
\newblock \textcolor{black}{The path solver: a nommonotone stabilization scheme
  for mixed complementarity problems}.
\newblock \emph{Optimization Methods and Software}, 5\penalty0 (2):\penalty0
  123--156, jan 1995.
\newblock \doi{10.1080/10556789508805606}.

\bibitem[Facchinei and Pang(2007)]{facchinei2007finite}
F.~Facchinei and J.~Pang.
\newblock \emph{Finite-Dimensional Variational Inequalities and Complementarity
  Problems}.
\newblock Springer Series in Operations Research and Financial Engineering.
  Springer New York, 2007.
\newblock ISBN 9780387218151.

\bibitem[Ferris and Munson(1999)]{Ferris1999}
M.~C. Ferris and T.~S. Munson.
\newblock \textcolor{black}{Interfaces to {PATH} 3.0: Design, Implementation
  and Usage}.
\newblock In \emph{Computational Optimization}, pages 207--227. Springer {US},
  1999.

\bibitem[Henrion et~al.(2012)Henrion, Outrata, and Surowiec]{henrion12}
R.~Henrion, J.~Outrata, and T.~Surowiec.
\newblock Analysis of m-stationary points to an epec modeling oligopolistic
  competition in an electricity spot market.
\newblock \emph{ESAIM: Control, Optimisation and Calculus of Variations},
  18\penalty0 (2):\penalty0 295–317, 2012.

\bibitem[Hu and Fukushima(2013)]{HuFu2013}
M.~Hu and M.~Fukushima.
\newblock Existence, uniqueness, and computation of robust nash equilibria in a
  class of multi-leader-follower games.
\newblock \emph{SIAM Journal on Optimization}, 23\penalty0 (2):\penalty0
  894--916, 2013.

\bibitem[Hu and Ralph(2007)]{HR07}
X.~Hu and D.~Ralph.
\newblock Using epecs to model bilevel games in restructured electricity
  markets with locational prices.
\newblock \emph{Operations Research}, 55\penalty0 (5):\penalty0 809--827, 2007.

\bibitem[Janin(1984)]{Janin1984}
R.~Janin.
\newblock \emph{Directional derivative of the marginal function in nonlinear
  programming}, pages 110--126.
\newblock Springer Berlin Heidelberg, Berlin, Heidelberg, 1984.
\newblock ISBN 978-3-642-00913-6.
\newblock \doi{10.1007/BFb0121214}.
\newblock URL \url{https://doi.org/10.1007/BFb0121214}.

\bibitem[Kim and Ferris(2019)]{Kim2019}
Y.~Kim and M.~C. Ferris.
\newblock \textcolor{black}{Solving equilibrium problems using extended
  mathematical programming}.
\newblock \emph{Mathematical Programming Computation}, 11\penalty0
  (3):\penalty0 457--501, mar 2019.
\newblock \doi{10.1007/s12532-019-00156-4}.

\bibitem[Koh and Shepherd(2010)]{Koh10}
A.~Koh and S.~Shepherd.
\newblock Tolling, collusion and equilibrium problems with equilibrium
  constraints.
\newblock 2010.
\newblock ISSN 1825-3997.

\bibitem[Kulkarni and Shanbhag(2014)]{Kulkarni2014}
A.~A. Kulkarni and U.~V. Shanbhag.
\newblock \textcolor{black}{A Shared-Constraint Approach to Multi-Leader
  Multi-Follower Games}.
\newblock \emph{Set-Valued and Variational Analysis}, 22\penalty0 (4):\penalty0
  691--720, aug 2014.
\newblock \doi{10.1007/s11228-014-0292-5}.

\bibitem[Kulkarni and Shanbhag(2015)]{Kulkarni2015}
A.~A. Kulkarni and U.~V. Shanbhag.
\newblock \textcolor{black}{An Existence Result for Hierarchical Stackelberg
  v/s Stackelberg Games}.
\newblock \emph{{IEEE} Transactions on Automatic Control}, 60\penalty0
  (12):\penalty0 3379--3384, dec 2015.
\newblock \doi{10.1109/tac.2015.2423891}.

\bibitem[Leyffer and Munson(2010)]{LeyfferMunson10}
S.~Leyffer and T.~Munson.
\newblock Solving multi-leader-common-follower games.
\newblock \emph{Optimization Methods and Software}, 25\penalty0 (4):\penalty0
  601--623, 2010.

\bibitem[Monderer and Shapley(1996)]{MONDERER1996}
D.~Monderer and L.~S. Shapley.
\newblock Potential games.
\newblock \emph{Games and Economic Behavior}, 14\penalty0 (1):\penalty0 124 --
  143, 1996.
\newblock ISSN 0899-8256.

\bibitem[Nikaid{\^o} and Isoda(1955)]{nikaido1955note}
H.~Nikaid{\^o} and K.~Isoda.
\newblock Note on non-cooperative convex game.
\newblock \emph{Pacific Journal of Mathematics}, 5\penalty0 (5):\penalty0
  807--815, 1955.

\bibitem[Pang and Fukushima(2005)]{Pang2005}
J.-S. Pang and M.~Fukushima.
\newblock Quasi-variational inequalities, generalized nash equilibria, and
  multi-leader-follower games.
\newblock \emph{Computational Management Science}, 2\penalty0 (1):\penalty0
  21--56, Jan 2005.
\newblock ISSN 1619-6988.

\bibitem[Qi and Sun(1993)]{Qi1993}
L.~Qi and J.~Sun.
\newblock A nonsmooth version of newton's method.
\newblock \emph{Mathematical Programming}, 58\penalty0 (1):\penalty0 353--367,
  Jan 1993.
\newblock ISSN 1436-4646.

\bibitem[Scheel and Scholtes(2000)]{ScheelScholtes00}
H.~Scheel and S.~Scholtes.
\newblock Mathematical programs with complementarity constraints: Stationarity,
  optimality, and sensitivity.
\newblock \emph{Mathematics of Operations Research}, 25\penalty0 (1):\penalty0
  1--22, 2000.
\newblock ISSN 0364765X, 15265471.

\bibitem[Sherali(1984)]{Sherali84}
H.~D. Sherali.
\newblock A multiple leader stackelberg model and analysis.
\newblock \emph{Operations Research}, 32\penalty0 (2):\penalty0 390--404, 1984.

\bibitem[Steffensen and Th\"unen(2019)]{Steffensen2019P}
S.~Steffensen and A.~Th\"unen.
\newblock An explicit nash equilibrium to a multi-leader-follower game.
\newblock In \emph{PAMM-Proc. Appl. Math. Mech.} Wiley-VCH Verlag GmbH \& Co.
  KGaA, Weinheim, 2019.

\bibitem[Steffensen and Ulbrich(2010)]{Sonja10}
S.~Steffensen and M.~Ulbrich.
\newblock A new relaxation scheme for mathematical programs with equilibrium
  constraints.
\newblock \emph{SIAM Journal on Optimization}, 20\penalty0 (5):\penalty0
  2504--2539, 2010.

\bibitem[Su(2005)]{SU2005}
C.-L. Su.
\newblock A sncp method for solving equilibrium problems with equilibrium
  constraints.
\newblock \emph{Computing in Economics and Finance 150}, 2005.

\bibitem[Su(2007)]{Su2007}
C.-L. Su.
\newblock \textcolor{black}{Analysis on the forward market equilibrium model}.
\newblock \emph{Operations Research Letters}, 35\penalty0 (1):\penalty0 74--82,
  jan 2007.
\newblock \doi{10.1016/j.orl.2006.01.006}.

\bibitem[Wang et~al.(2008)Wang, Chiu, and Lui]{Wang08}
J.~H. Wang, D.~M. Chiu, and J.~C.~S. Lui.
\newblock A game-theoretic analysis of the implications of overlay network
  traffic on ips peering.
\newblock \emph{Computer Networks}, 2008.

\end{thebibliography}

\end{document}